\documentclass{amsart}

\usepackage{amsmath, amsfonts, amssymb, amsthm, amscd, graphicx, float, epstopdf, tabu
}

\allowdisplaybreaks[4]

\usepackage[pdfpagemode={UseOutlines},bookmarks=true,bookmarksopen=true, bookmarksopenlevel=0,bookmarksnumbered=true,hypertexnames=false, colorlinks,linkcolor={blue},citecolor={blue},urlcolor={blue}, pdfstartview={FitV},unicode,breaklinks=true,backref=page]{hyperref}

\newtheorem{thm}{Theorem}[section]
\newtheorem{remark}[thm]{Remark}
\newtheorem{notation}[thm]{Notation}

\newtheorem{defn}[thm]{Definition}
\newtheorem{lem}[thm]{Lemma}
\newtheorem{prop}[thm]{Proposition}
\newtheorem{cor}[thm]{Corollary}

\begin{document}

\title{Rank $n$ swapping algebra for $\operatorname{PGL}_n$ Fock--Goncharov $\mathcal{X}$ moduli space}


%    author one information
\author{Zhe Sun}
\address{Department of mathematics, University of Luxembourg}
\email{sunzhe1985@gmail.com}
\thanks{The research leading to these results has received funding from the European Research Council under the {\em European Community}'s seventh Framework Programme (FP7/2007-2013)/ERC {\em Grant agreement no FP7-246918}.}

%    \subjclass is required.

\date{}

\dedicatory{To William Goldman on the occasion of his sixtieth birthday}

%    Abstract is required.
\begin{abstract}
The {\em rank $n$ swapping algebra} is a Poisson algebra defined on the set of ordered pairs of points of the circle using linking numbers, whose geometric model is given by a certain subspace of $(\mathbb{K}^n \times \mathbb{K}^{n*})^r/\operatorname{GL}(n,\mathbb{K})$. For any ideal triangulation of $D_k$---a disk with $k$ points on its boundary, using determinants, we find an injective Poisson algebra homomorphism from the fraction algebra generated by the Fock--Goncharov coordinates for $\mathcal{X}_{\operatorname{PGL}_n,D_k}$ to the rank $n$ swapping multifraction algebra for $r=k\cdot(n-1)$ with respect to the (Atiyah--Bott--)Goldman Poisson bracket and the swapping bracket. This is the building block of the general surface case. Two such injective Poisson algebra homomorphisms related to two ideal triangulations $\mathcal{T}$ and $\mathcal{T}'$ are compatible with each other under the flips.
\end{abstract}

\maketitle


\section{Introduction}
We study the {\em Goldman Poisson structure} using circle, linking numbers, determinants and ratios. 
\subsection{Background and motivation}
Let $\mathcal{R}_{G,S}$ be the space of gauge equivalence classes of flat connections on a fixed principal $G$-bundle over $S$, where $G$ is a reductive Lie group and $S$ is a connected oriented closed Riemann surface of genus $g>1$. In early 80s, Atiyah and Bott \cite{atiyah1983yang} constructed a symplectic structure $\omega$ on $\mathcal{R}_{G,S}$ by symplectic reduction from infinite dimensional symplectic manifold $\mathcal{M}_{G,S}$ of all such connections via the moment map given by the curvature. From another point of view, the space $\mathcal{R}_{G,S}$ is $\operatorname{Hom}(\pi_1(S),G)/G$ using the monodromy representation of $\pi_1(S)$ with respect to the connection, where $G$ acts by conjugation. Then Goldman \cite{G84} identified the tangent space of $\operatorname{Hom}(\pi_1(S),G)/G$ with the group cohomology $H^1(\pi_1(S), \mathfrak{g})$ and interpret the symplectic structure $\omega$ in terms of the intersection pairings on the surface $S$---the cup product. This construction has been extended to the case where the Riemann surface $S$ has finitely many boundary components in \cite{AM95} \cite{GHJW97} and references therein, even with marked points on the boundary in \cite{FR98}.

When $S$ is a closed Riemann surface of genus $g>1$, there is a special connected component $H_n(S)$ of $\mathcal{R}_{\operatorname{PGL}(n,\mathbb{R}),S}$, containing all the $n$-Fuchsian representations, called {\em Hitchin component} \cite{H92}. Here an $n$-Fuchsian representation is the composition of a discrete faithful representation from $\pi_1(S)$ to $\operatorname{PSL}(2,\mathbb{R})$ with an irreducible representation from $\operatorname{PSL}(2,\mathbb{R})$ to $\operatorname{PGL}(n,\mathbb{R})$. The Hitchin component $H_n(S)$ is so nice that the quotient in the sense of geometric invariant theory \cite{MK94} is the same as its topological quotient. 

Let $E$ be a $n$-dimensional vector space and let $\Omega$ be a non-zero volume form of $E$. A {\em flag} $F$ for $\operatorname{PGL}_n$ is a nested collection of vector subspaces in $E$
\[\{0\subset F_1 \subset \cdots \subset F_{n-1}\subset F_n=E\;|\; \dim F_i=i \}\]
equipped with the volume form $\Omega$. The {\em flag variety} $\mathcal{B}$ is the space of all flags.
Labourie and Guichard \cite{Gu08}\cite{L06} identified each element $\rho$ in the Hitchin component $H_n(S)$ with a $\rho$-equivariant hyperconvex Frenet curve $\xi_\rho$ from the boundary at infinity $\partial_{\infty}(\pi_1(S))$ of $\pi_1(S)$ to the flag variety $\mathcal{B}(\mathbb{R})$ up to diagonal action by projective transformations. This identification allows us to study the Goldman Poisson structure on $H_n(S)$ by studying the Goldman Poisson structure on the space {\em $\mathcal{FR}_n$} of hyperconvex Frenet curves up to projective transformations. We write $\xi_\rho$ as a $(n-1)$-tuple $(\xi_\rho^1,\cdots, \xi_\rho^{n-1})$, where $\xi_\rho^{i}$ takes values in the Grassmannian $G_i(\mathbb{R}^n)$ for $i=1,\cdots, n-1$. Let $\tilde{\xi}^1_\rho$ ($\tilde{\xi}^{n-1}_\rho$ resp.) be any lift of $\xi_\rho$ ($\xi^{n-1}_\rho$ resp.) with the values in $\mathbb{R}^n$ ($\mathbb{R}^{n*}$ resp.). For any four distinct points $\textbf{x},\textbf{y},\textbf{z},\textbf{t}$ in $\partial_{\infty}(\pi_1(S))$, Labourie \cite{L07} defined a special function on the Hitchin component $H_n(S)$, called the {\em weak cross ratio}, defined as follows:
$$	\mathbb{B}_{\rho}(\textbf{x},\textbf{y},\textbf{z},\textbf{t}) = \frac{\left<\left.\tilde{\xi}^1_\rho(\textbf{x})\right\vert \tilde{\xi}^{n-1}_\rho(\textbf{z})\right> }{\left<\left.\tilde{\xi}^1_\rho(\textbf{x})\right\vert \tilde{\xi}^{n-1}_\rho(\textbf{t})\right>} \cdot \frac{\left<\left.\tilde{\xi}^1_\rho(\textbf{y})\right\vert \tilde{\xi}^{n-1}_\rho(\textbf{t})\right>}{\left<\left.\tilde{\xi}^1_\rho(\textbf{y})\right\vert \tilde{\xi}^{n-1}_\rho(\textbf{z})\right>}.
$$

Investigating the algebraic nature of weak cross ratios, Labourie \cite{L18} introduced the swapping algebra to characterize the Goldman Poisson structure on $H_n(S)$ for any $n>1$ and the second Adler--Gel'fand--Dickey Poisson structure \cite{Ma78}\cite{Di97} (and references therein) via Drinfel'd--Sokolov reduction \cite{DS81} on the space $Opers_n$ of $\operatorname{SL}(n,\mathbb{R})$-opers with trivial holonomy. Let us recall the swapping algebra as follows. 

We represent an {\em ordered pair} (of points) $(x, y)$ of a given set $\mathcal{P} \subseteq S^1$ by the expression $xy$, and we consider the commutative ring $$\mathcal{Z}(\mathcal{P}) := \mathbb{K}[\{xy\}_{ x,y \in \mathcal{P}}]/\left(\{xx\}_{ x \in \mathcal{P}}\right)$$ over a field $\mathbb{K}$ of characteristic zero. The ring $\mathcal{Z}(\mathcal{P})$ is equipped with the Poisson bracket $\{\cdot,\cdot\}$, called the {\em swapping bracket}, defined by extending to $\mathcal{Z}(\mathcal{P})$ the following formula on arbitrary generators $rx, sy$:
\begin{equation}
\{rx, sy\} = \mathcal{J}(rx, sy) \cdot ry \cdot sx ,
\end{equation}
using Leibniz's rule.
We define the {\em linking number} $\mathcal{J}(rx, sy) \in \{0, \pm1, \pm \frac{1}{2}\}$ on $S^1$ as in Figure \ref{swapswap1} which only depends on the corresponding position of the four points.
The {\em swapping algebra} of $\mathcal{P}$ is $(\mathcal{Z}(\mathcal{P}), \{\cdot , \cdot\})$. Then the {\em swapping multifraction algebra} $\left(\mathcal{B}(\mathcal{P}),\{\cdot,\cdot\}\right)$ is the sub-fraction algebra of the swapping algebra $\left(\mathcal{Z}(\mathcal{P}),\{\cdot,\cdot\}\right)$ generated by {\em cross fractions} like $\frac{xz}{xt}\cdot \frac{yt}{yz}$. 
By considering the homomorphism 
\[\tau :\mathcal{B}(\mathcal{P}) \rightarrow C^\infty\left(H_n(S)\right)\]
 that sends $\frac{xz}{xt}\cdot \frac{yt}{yz}$ to $\mathbb{B}_{\rho}(\textbf{x},\textbf{y},\textbf{z},\textbf{t})$, Labourie \cite{L18} showed that $\tau$ is ``asymptotically Poisson'' with respect to the swapping bracket and the Goldman Poisson bracket. 

However $\tau$ is not injective. To make the kernel of $\tau$ smaller, the {\em rank $n$ swapping algebra} $(\mathcal{Z}_n(\mathcal{P}),\{\cdot,\cdot\})$ is introduced in \cite{Su17}. Here $(\mathcal{Z}_n(\mathcal{P}),\{\cdot,\cdot\})$ is the quotient of the swapping algebra $(\mathcal{Z}(\mathcal{P}), \{\cdot , \cdot\})$ by the Poisson ideal $R_n(\mathcal{P})$ generated by $$\left\{ \det \left(x_i y_j\right)_{i,j=1}^{n+1} \in \mathcal{Z}(\mathcal{P}) \; |  \;   x_1, \cdots, x_{n+1}, y_1,\cdots, y_{n+1} \in \mathcal{P} \right\}.$$

The geometric model for $\mathcal{Z}_n(\mathcal{P})$ in \cite[Section 4]{Su17} arises naturally from the classical geometric invariant theory \cite{W39}. When $\mathcal{P}=\{x_1,\cdots,x_r\}$, we associate a pair $(\mathfrak{a}_i,\mathfrak{b}_i)\in \mathbb{K}^n \times \mathbb{K}^{n*}$ to each $x_i$ for $i=1,\cdots,r$. We consider the space $D_{n,r} = (\mathbb{K}^n \times \mathbb{K}^{n*})^r$ of $r$ vectors $\mathfrak{a}_1,\cdots,\mathfrak{a}_r$ in $\mathbb{K}^n$ and $r$ covectors $\mathfrak{b}_1,\cdots,\mathfrak{b}_r$ in $\mathbb{K}^{n*}$. For any $g \in \operatorname{GL}(n,\mathbb{K})$, the action of $g$ on the vector $\mathfrak{a}_i$ is the left multiplication by $g$, the action of $g$ on the covector $\mathfrak{b}_i$ is the right multiplication by $g^{-1}$. We define the product between a vector $\mathfrak{a}_i$ in $\mathbb{K}^n$ and a covector $\mathfrak{b}_j$ in $\mathbb{K}^{n*}$ by $\left<\mathfrak{a}_i|\mathfrak{b}_j\right> := \mathfrak{b}_j(\mathfrak{a}_i)$, which is $\operatorname{GL}(n,\mathbb{K})$ invariant. Then we associate each $\left<\mathfrak{a}_i|\mathfrak{b}_j\right>$ to each pair $x_i x_j \in \mathcal{Z}_n(\mathcal{P})$. The geometric model for $(\mathcal{Z}_n(\mathcal{P}),\{\cdot,\cdot\})$ is 
\begin{equation}
\label{equation:gm}
\mathcal{D}_{n,r}=\{(\mathfrak{a}_1,\mathfrak{b}_1,\cdots, \mathfrak{a}_r,\mathfrak{b}_r)\in D_{n,r}\;|\; <\mathfrak{a}_i|\mathfrak{b}_i>=0, i=1,\cdots,r\}/\operatorname{GL}(n,\mathbb{K}),
\end{equation}
which is also equipped with the swapping bracket. 

The {\em rank $n$ swapping multifraction algebra} $\left(\mathcal{B}_n(\mathcal{P}),\{\cdot,\cdot\}\right)$ is the sub-fraction algebra of the rank $n$ swapping algebra $\left(\mathcal{Z}_n(\mathcal{P}),\{\cdot,\cdot\}\right)$ generated by cross fractions. 
Then the homomorphism $\tau$ is changed into 
\[\tau_n :\mathcal{B}_n(\mathcal{P}) \rightarrow C^\infty\left(H_n(S)\right)\]
 that sends $\frac{xz}{xt}\cdot \frac{yt}{yz}$ to $\mathbb{B}_{\rho}(\textbf{x},\textbf{y},\textbf{z},\textbf{t})$. However $\tau_n$ is still not injective because of the $\pi_1(S)$ invariance of the weak cross ratios.

It motivates us to consider an injective Poisson homomorphism $\theta$ that sends a coordinate fraction ring of $H_n(S)$ to $\mathcal{B}_n(\mathcal{P})$ in Definition \ref{defnhom}. Our crucial point for defining such homomorphism is that we characterize a pairing between a vector and a covector by a $(n\times n)$-determinant in $\mathcal{Z}_n(\mathcal{P})$ instead of an ordered pair in $\mathcal{Z}_n(\mathcal{P})$. It turns out that the Fock--Goncharov coordinates for $\mathcal{X}_{\operatorname{PGL}_n,\hat{S}}$ \cite{FG06} work well, even for their corresponding quantized algebras \cite{Su}. The Poisson homomorphism in \cite[Theorem 10.7.2]{L18} for $Opers_n$ is still Poisson after replacing an ordered pair in $\mathcal{Z}_n(\mathcal{P})$ by a $(n\times n)$-determinant as shown in Section \ref{section:opers}. 

\emph{Instead of understanding the Goldman Poisson structure as the cup product on the surface, the rank $n$ swapping algebra provides another natural description using circle, linking numbers, determinants, ratios and classical geometric invariant ring.
}

\subsection{The main result}
\emph{We use $\textbf{x}$ for a vertex on the surface to distinguish it from an element $x$ of $\mathcal{P}$ throughout this paper.
}

Let $D_k$ be a disk $D$ with $k \geq 3$ points $m_b=\{\textbf{s}\prec \textbf{w}\prec \cdots \prec \textbf{t} \prec \textbf{s}\}$ on $\partial D$, where $\prec$ defines a cyclic order with respect to the anticlockwise orientation on the circle. 
In this case $\mathcal{X}_{\operatorname{PGL}_n,D_k} \cong \mathcal{B}^k/\operatorname{PGL}_n$ with respect to the diagonal action of $\operatorname{PGL}_n$.
Given an ideal triangulation $\mathcal{T}$ of $D_k$, the {\em $n$-triangulation} $\mathcal{T}_n$ is a subdivision of $\mathcal{T}$ such that each triangle of $\mathcal{T}$ is divided into $n^2$ triangles as in Figure \ref{figurekgon}. Fock and Goncharov (Definition \ref{definition:fgcoor}) parameterize $\mathcal{X}_{\operatorname{PGL}_n,D_k}$ by assigning the coordinate $X_V$ to each vertex $V$ of certain subset of vertices of $\mathcal{T}_n$. Let $\mathcal{FX}(\mathcal{T}_n)$ be the fraction ring generated by these $\{X_V\}$ over the field $\mathbb{K}$.
The {\em rank $n$ Fock--Goncharov Poisson bracket} is given by 
\[\left\{X_V ,\; X_W \right\}_{n} =  \varepsilon_{V, W} \cdot X_V \cdot X_W,
\]
where by Figure \ref{figurekgon}
\begin{equation}
\label{equation:eps}
\varepsilon_{VW} =  \#\{\;arrows \;from\; V \; to \; W\;\} -\# \{\;arrows \;from\; W \; to \; V\;\}.
\end{equation} 
The {\em rank $n$ Fock--Goncharov algebra} is $(\mathcal{FX}(\mathcal{T}_n),\{\cdot,\cdot\}_n)$.

Given $\xi \in \mathcal{B}^k$, for any $\textbf{r}\in m_b$, we choose a basis $\{\textbf{r}_1, \cdots, \textbf{r}_n\}$ of the flag $\xi(\textbf{r})$ such that $\textbf{r}_1, \cdots, \textbf{r}_i$ span the $i$ dimensional subspace $\xi^i(\textbf{r})$ of $\xi(\textbf{r})$. Then we choose a covector $\textbf{r}_i^c$ such that $\left<\textbf{r}_i|\textbf{r}_i^c\right>=0$ for $i=1,\cdots,n-1$.
For $\mathcal{X}_{\operatorname{PGL}_n,D_k}$, the key observation for relating the rank $n$ Fock--Goncharov algebra to the rank $n$ swapping multifraction algebra is the following:

{\em Any $\textbf{r}\in m_b$ and $i\in \{1,\cdots,n-1\}$ provide us a pair 
\[(\textbf{r}_i, \textbf{r}_i^c) \in E \times E^*\;\;such\;\;that\;\; \left<\textbf{r}_i|\textbf{r}_i^c\right>=0\]
which embeds $\mathcal{X}_{\operatorname{PGL}_n,D_k}$ into the subspace of $(E\times E^*)^{k(n-1)}/\operatorname{GL}_n$ subject to $\left<\textbf{r}_i|\textbf{r}_i^c\right>=0$. The induced Poisson structure on $\mathcal{X}_{\operatorname{PGL}_n,D_k}$ from the swapping bracket does not depend on the choice of bases of the flags.}

Here each $\textbf{r} \in m_b$ is related to $n-1$ elements in $E\times E^*$.
Therefore we define 
\[\mathcal{P}= \{s_{n-1}\prec \cdots \prec s_{1} \prec w_{n-1} \prec \cdots \prec w_{1} \prec \cdots \prec t_{n-1} \prec  \cdots \prec t_{1} \prec s_{n-1}\}\]
on $S^1$, where each $\textbf{r}\in m_b$ corresponds to $n-1$ anticlockwise ordered points $r_{n-1},\dots, r_1$ nearby in $\mathcal{P}$ as in Figure \ref{Figure:homo}.

Suppose that $V$ is a vertex of the $n$-triangulation $\mathcal{T}_n$ related to the marked triangle $(\textbf{x},\textbf{y},\textbf{z})$ of the ideal triangulation $\mathcal{T}$ and a triple of non-negative integers $(m,l,p)$ with $m+l+p=n$. Choose some bases 
\[\{\textbf{x}_1,\cdots,\textbf{x}_n\},\;\;\{\textbf{y}_1,\cdots,\textbf{y}_n\},\;\;\{\textbf{z}_1,\cdots,\textbf{z}_n\}.\]
for the flags $\xi(\textbf{x})$, $\xi(\textbf{y})$, $\xi(\textbf{z})$ respectively. Fix a volume form $\Omega$ of $E$. Let
\[\Delta_V =  \Omega\left(\textbf{x}_1\wedge \cdots \wedge \textbf{x}_{m}\wedge \textbf{y}_1\wedge \cdots \wedge \textbf{y}_{l} \wedge \textbf{z}_1\wedge \cdots \wedge \textbf{z}_{p}\right).\]
Let $\mathcal{FA}_n$ be the fraction ring generated by all these determinants with the fixed bases of flags.

For any $d>1$ and any $x_1,\cdots,x_d, y_1, \cdots, y_d \in \mathcal{P}$, let us adopt the notation
\begin{equation}
\label{equation:notation1}
\Delta\left((x_1,\cdots,x_d), (y_1, \cdots, y_d)\right): = \det   \left(\begin{array}{cccc}
       x_1 y_1 & \cdots & x_1 y_d \\
       \cdots & \cdots & \cdots \\
       x_d y_1 & \cdots & x_d y_d
     \end{array}\right) \in \mathcal{Z}_n(\mathcal{P}).
\end{equation}

Fixed a choice of distinct $u_1,\cdots, u_n \in \mathcal{P}$, the homomorphism $\chi_n$ from $\mathcal{FA}_n$ to $\mathcal{Q}_n(\mathcal{P})$ is defined by extending the following formula on the generators to $\mathcal{FA}_n$ using Leibniz's rule
\begin{equation}
\label{equation:chi}
\chi_n(\Delta_V) = \Delta\left(\left(x_1,\cdots, x_m, y_1,\cdots, y_l, z_1,\cdots, z_p\right), \left(u_1,\cdots u_n\right)\right).
\end{equation}
We define the homomorphism $\theta_{\mathcal{T}_n}$ from $\mathcal{FX}(\mathcal{T}_n)$ to $\mathcal{B}_n(\mathcal{P})$ by restricting the homomorphism $\chi_n$ to the fraction ring $\mathcal{FX}(\mathcal{T}_n)$. Then the homomorphism $\theta_{\mathcal{T}_n}$ does not depend on the choice of bases of flags and the choice of distinct $u_1,\cdots, u_n \in \mathcal{P}$. More explicitly, we have 
\[\theta_{\mathcal{T}_n}(X_V) =\chi_n(X_V)= \chi_n( \prod_W \Delta_W^{\varepsilon_{VW}}).\] 
\begin{thm}
{\sc[Main result Theorem \ref{thmm}]}
Given an ideal triangulation $\mathcal{T}$ of $D_k$, there is an injective Poisson homomorphism $\theta_{\mathcal{T}_n}$ from the rank $n$ Fock--Goncharov algebra for the moduli space $\mathcal{X}_{\operatorname{PGL}_n,D_k}$ to the rank $n$ swapping multifraction algebra $(\mathcal{B}_n(\mathcal{P}), \{\cdot,\cdot\})$, with respect to the Goldman Poisson bracket and the swapping bracket.
\end{thm}
The above theorem generalizes the result for $n=2$ in \cite{Su17}, for $n=3$ in Chapter $3$ of \cite{Su14}. Combining this theorem with the main result in \cite{L18}, we again relate the Fock--Goncharov Poisson structure with the Goldman Poisson structure.

To prove the main result, we introduce the quotient of two $(n\times n)$-determinants $$\frac{\Delta\left(\left(x_1,\cdots, x_{n}\right), \left(u_1,\cdots, u_n\right)\right)}{\Delta\left(\left(y_1, \cdots,y_{n}\right), \left(u_1,\cdots, u_n\right)\right)}$$ with the same right side $n$-tuple of distinct points $(u_1,\cdots,u_n)$, called the $(n\times n)$-{\em determinant ratio} in the field of fractions $\mathcal{Q}_n(\mathcal{P})$ of $\mathcal{Z}_n(\mathcal{P})$. It has the nice property that it does not depend the choice of $(u_1,\cdots,u_n)$, due to the $R_n(\mathcal{P})$ relations. Then we realize that any $\theta_{\mathcal{T}_n}(X_V)$ is a product of two or three $(n\times n)$-determinant ratios which can be represented by two or three oriented edges of $\mathcal{T}_n$ as in Figure \ref{imte}. Such $(n \times n)$-determinant ratio is called {\em oriented edge ratio}. As a consequence $\theta_{\mathcal{T}_n}$ does not depend on the choice of $u_1,\cdots, u_n \in \mathcal{P}$. 

By Lemma \ref{swcal}, the swapping bracket
\[\{ab,\Delta((x_1, \cdots, x_{n}), (y_1,\cdots, y_{n}))\}=\Delta^R(ab)=\Delta^L(ab)\]
can be expressed in two different ways regarding to the right or left side of $\overrightarrow{ab}$ in Figure \ref{swapabdet1}. Then we compute the swapping bracket between two $(n\times n)$-determinants in our main proposition \ref{propxyz}, which is the most technical part for proving the main theorem. Let us fix some notations 
\begin{equation}
\label{equation:notation2}
[A,B]:=\frac{\{A,B\}}{AB},\;\; w^i:=w_1,\cdots,w_i.
\end{equation}
We stress the fact that the formula
\begin{eqnarray*}
&&\left[\Delta\left(\left(x^m, y^l, z^p \right), \left(v^n\right)\right), \Delta\left(\left(x^{m'}, y^{l'}, z^{p'}\right), \left(u^n\right)\right)\right]
\\&=& \frac{1}{2}\cdot \min\{m,m'\}- \frac{1}{2}\cdot \min\{l, l'\} - \frac{1}{2}\cdot \min\{p, p'\}
\end{eqnarray*}
in Proposition \ref{propxyz} strictly depends on the cyclic order in Figure \ref{Figure:swapother4} and the condition (*) $l\geq l'$ or $p \leq p'$ is strict. Essentially, the $+$ and $-$ sign before $\frac{1}{2}\cdot \min$ is due to our cyclic order.  Then we obtain Proposition \ref{propdr}, which shows that the $[\cdot,\cdot]$ bracket between any two oriented edge ratios belongs to $\{-\frac{1}{2},0,\frac{1}{2}\}$ and only depends the corresponding positions of two oriented edge ratios as in Figure \ref{der1}, \ref{der2}. Our oriented edge ratios correspond to the generalized Kashaev coordinates \cite{K98} \cite{Ki16}, but their Poisson bracket is different from the swapping bracket for two oriented edges lying on two different ideal triangles. 
In the proof of Proposition \ref{propdr}, we choose the right side $n$-tuples wisely for the $(n\times n)$-determinant ratios in each case in order to use Proposition \ref{propxyz} under the condition (*). Finally by checking all the cases, we finish the proof of the main theorem. 
\subsection{Compatible}
For $D_k$, we can transform any ideal triangulation $\mathcal{T}$ to any other ideal triangulation $\mathcal{T}'$ by a finite sequence of flips. In Proposition \ref{proposition:T}, we prove that the corresponding two injective Poisson homomorphisms $\theta_{\mathcal{T}_n}$ and $\theta_{\mathcal{T}_n'}$ are compatible by a generalized Pl\"ucker relation in $\mathcal{Z}_n(\mathcal{P})$. It is realted to the result for $n=2$ in \cite[Lemma 5.5]{Su17} where the cross fractions are used to define the homomorphism $\theta_{\mathcal{T}_n}$. As a corollary, the rank $n$ Fock--Goncharov Poisson bracket does not depend on the ideal triangulation $\mathcal{T}$. Note that these properties only become possible after dividing $(\mathcal{Z}(\mathcal{P}),\{\cdot,\cdot\})$ by $R_n(\mathcal{P})$. 
\subsection{From disk to surface}
\label{subsection:ds}
Let $\hat{S}= (S=S_{g,m},\emptyset)$ with $2g-2+m>0$. In this case, we obtain a homotopy equivalent surface $S'$ by shrinking holes on $S$ to punctures. The ideal triangulation of $\hat{S}$ is the ideal triangulation of $S'$ with vertices at the punctures. Let $\widetilde{\mathcal{T}}_n$ be all the lifts of the $n$-triangulation $\mathcal{T}_n$ into the universal cover of the surface $S'$. The {\em Farey set} $\mathcal{F}_\infty(S)$ is the countably infinite collection of vertices of $\widetilde{\mathcal{T}}$, equipped with a cyclic order on the boundary at infinity $\partial_\infty \pi_1(S')$. Let $\mathcal{P}$ be a cyclic subset of $S^1$ obtained by splitting each point of $\mathcal{F}_\infty(S)$ into $n-1$ points nearby in $S^1$ as we did for $D_k$. By our main theorem \ref{mainresult}, the injective homomorphism $\theta_{\widetilde{\mathcal{T}_n}}$ from $\mathcal{FX}(\widetilde{\mathcal{T}}_n)$ to $\mathcal{B}_n(\mathcal{P})$ is Poisson with respect to the rank $n$ Fock--Goncharov Poisson bracket and the swapping bracket. Thus the swapping bracket identifies with the rank $n$ Fock--Goncharov Poisson bracket on the universal cover $\operatorname{Conf}_{\mathcal{F}_\infty(S),n} \cong \mathcal{B}^{\mathcal{F}_\infty(S)}/\operatorname{PGL}_n$. Moreover, $\pi_1(S)$ acts on both $\mathcal{FX}(\widetilde{\mathcal{T}}_n)$ and $\theta_{\widetilde{\mathcal{T}_n}}(\mathcal{FX}(\widetilde{\mathcal{T}}_n))$ through the $\pi_1(S)$ action on the Farey set $\mathcal{F}_\infty(S)$, thus $\theta_{\widetilde{\mathcal{T}}_n}$ is $\pi_1(S)$-equivariant with respect to these actions. By \cite[Lemma 1.1]{FG06}, $\mathcal{X}_{\operatorname{PGL}_n,\hat{S}}\cong \operatorname{Conf}_{\mathcal{F}_\infty(S),n}^{\pi_1(S)}$. Then the rank $n$ Fock--Goncharov Poisson bracket on $\mathcal{X}_{\operatorname{PGL}_n,\hat{S}}$ is induced by the $\pi_1(S)$-equivariant homomorphism $\theta_{\widetilde{\mathcal{T}}_n}$ from the swapping bracket. This construction also works for $\hat{S}= (S,m_b)$ where $m_b \subset \partial S$ is a finite set since the cyclic order on the boundary at infinity $\partial_\infty \pi_1(S')$ induces a cyclic order on every lift of a boundary component containing marked points in $m_b$.
\subsection{From cross fractions to $(n\times n)$-determinant ratios}
Instead of characterizing the weak cross ratio by the cross fraction in the homomorphism $\tau_n$, we use a product of two $(n\times n)$-determinant ratios to characterize the weak cross ratio. By Theorem \ref{theorem:cfd}, such characterization is compatible with the former with respect to the swapping bracket. 
\subsection{Summary, further development}
Using $(n\times n)$-determinant ratios instead of cross fractions, we provide the following understanding of ``the space of all cross ratios'' proposed by Labourie. Using the swapping bracket, we define the Poisson structure on a subspace of $(\mathbb{K}^n \times \mathbb{K}^{n*})^{\#\mathcal{P}}/\operatorname{GL}(n,\mathbb{K})$. It induces a Poisson bracket $\omega_{\mathcal{SW}}$ on the sub fraction ring $\mathcal{DR}(\mathcal{FR}_n)$ of functions of $\mathcal{FR}_n$ (space of hyperconvex Frenet curves up to projective transformations) generated by all elements corresponding to $(n\times n)$-determinant ratios. By our main theorem, the Fock--Goncharov coordinate fraction ring of $\mathcal{X}_{\operatorname{PGL}_n,\hat{S}}$ is Poisson embedded into $\mathcal{DR}(\mathcal{FR}_n)$ with respect to the Goldman Poisson structure and $\omega_{\mathcal{SW}}$. On the other hand, combing \cite[Theorem 10.7.2]{L18} and Theorem \ref{theorem:cfd}, the fraction ring of acceptable
observables on the space $Opers_n$ of $\operatorname{SL}(n,\mathbb{R})$-opers with trivial holonomy is also Poisson embedded into $\mathcal{DR}(\mathcal{FR}_n)$ with respect to the second Gel'fand-Dickey Poisson structure and $\omega_{\mathcal{SW}}$.

In a forthcoming paper \cite{Su}, we define the quantized rank $n$ swapping algebra $\mathcal{Z}_n^q(\mathcal{P})$ generated over $\mathbb{K}_q = \mathbb{K}[q,q^{-1}]$ by non commutative indeterminates. Given any ideal triangulation $\mathcal{T}$, we give an injective homomorphism $\theta^q_{\mathcal{T}_n}$ from the quantized Fock--Goncharov coordinate fraction algebra \cite{FG09} for $\mathcal{X}_{\operatorname{PGL}_n,D_k}$ to the quantized rank $n$ swapping multifraction algebra $\mathcal{B}_n^q(\mathcal{P})$. Moreover, we show that any two homomorphisms $\theta^q_{\mathcal{T}_n}$ and $\theta^q_{\mathcal{T}'_n}$ are compatible using Skein relations.

Moreover, we suggest the following research directions.
\begin{enumerate}
\item In the conference organized by Goldman at Maryland University in 2016, Labourie \cite{L} described a compactification of the Hitchin component $H_n(S)$ using a tropical version of cross ratio and the rank $n$ swapping algebra. It is interesting to relate this compactification to the compactification of cluster $\mathcal{X}$ variety at infinity in \cite{FG16} through the injective Poisson homomorphism $\theta_{\mathcal{T}_n}$. 
\item It is interesting to investigate the Fock--Goncharov Poisson structures for the surface $S$ with bordered cusps as in \cite{CM17}, \cite{CMR17} via the rank $n$ swapping algebra, where they build very interesting links with Painlev\'e-type equations.
\end{enumerate}

\section{Rank $n$ swapping algebra}
\label{section:rnswap}
In this section, we recall the swapping algebra \cite{L18} and the rank $n$ swapping algebra \cite{Su17}. Lemma \ref{swcal} (\cite[Lemma 3.5, Remark 3.6]{Su17}) is the key technical formula to use for proving our main proposition \ref{propxyz}.
\subsection{Swapping algebra}
\label{sa}
\begin{defn}\label{defnlkn}
{\sc[linking number]}
Let $(r, x, s, y)$ be a quadruple of points in $\mathcal{P}\subset S^1$. We represent an ordered pair $(r, x)$ of $\mathcal{P}$ by the expression $rx$. Let $o$ be any point different from $r,x,s,y \in S^1$. 
Let $\sigma$ be a homeomorphism from $S^1\backslash o$ to $\mathbb{R}$ with respect to the anticlockwise orientation of $S^1$. Let $\triangle(a)= -1; 0; 1$ whenever $a < 0$; $a = 0$; $a > 0$ respectively.

The {\em linking number} between $rx$ and $sy$ is
\begin{equation}
\begin{aligned}
\label{equJ}
&\mathcal{J}(rx, sy) = \frac{1}{2}  \cdot \triangle(\sigma(r)-\sigma(x)) \cdot \triangle(\sigma(r)-\sigma(y)) \cdot \triangle(\sigma(y)-\sigma(x)) 
\\&- \frac{1}{2} \cdot \triangle(\sigma(r)-\sigma(x)) \cdot \triangle(\sigma(r)-\sigma(s)) \cdot \triangle(\sigma(s)-\sigma(x)).
\end{aligned}
\end{equation}

\end{defn}
\begin{figure}
\includegraphics[scale=0.5]{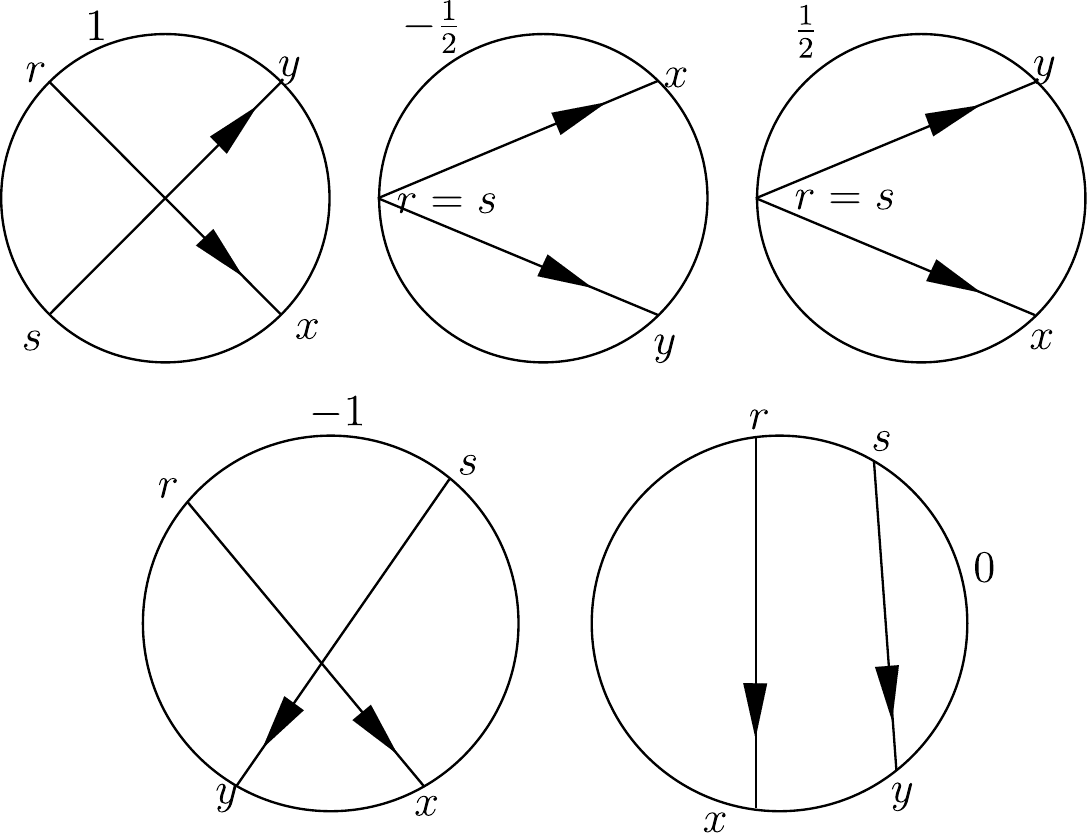}
\caption{Linking number $\mathcal{J}(rx, sy)$ between $rx$ and $sy$}
\label{swapswap1}
\end{figure}

In fact, the value of $\mathcal{J}(rx, sy)$ belongs to $\{0, \pm1, \pm \frac{1}{2}\}$, and does not depend on the choice of the point $o$ and depends only on the relative positions of $r,x,s,y$. In Figure \ref{swapswap1}, we describe five possible values of $\mathcal{J}(rx, sy)$.

For $\mathcal{P}$ a cyclic subset of $S^1$, we represent an ordered pair $(r, x)$ of $\mathcal{P}$ by the expression $rx$. Then we consider the associative commutative ring $$\mathcal{Z}(\mathcal{P}) := \mathbb{K}[\{xy\}_{\forall x,y \in \mathcal{P}}]/\{xx| \forall x \in \mathcal{P}\}$$ over a field $\mathbb{K}$ of characteristic zero, where $\{xy\}_{\forall x,y \in \mathcal{P}}$ are the set of variables.
Then we equip $\mathcal{Z}(\mathcal{P})$ with a swapping bracket, defined as follows.

\begin{defn}{\sc[swapping bracket \cite[$\alpha=0$ case]{L18}]}
The swapping bracket over $\mathcal{Z}(\mathcal{P})$ is defined by extending the following formula on arbitrary generators $rx$, $sy$ to $\mathcal{Z}(\mathcal{P})$ using {\em Leibniz's rule}
\[
\{rx, sy\} = \mathcal{J}(rx, sy) \cdot  \cdot ry \cdot sx.
\]
\end{defn}

By direct computations, Labourie proved the following theorem.
\begin{thm}
\label{swappoisson}
{\sc[Labourie \cite{L18}]} The swapping bracket is Poisson.
\end{thm}

Let $\mathcal{Q}(\mathcal{P})$ be the field of fractions of $\mathcal{Z}(\mathcal{P})$. We extend the swapping bracket to $\mathcal{Q}(\mathcal{P})$ by 
 \[\{rx,\frac{1}{sy}\}= -\frac{\{rx,sy\}}{sy^2}.\]

\begin{defn}
The {\em cross fraction} determined by $(x,y,z,t)$ is the element 
\[\frac{xz}{xt} \cdot \frac{yt}{yz}.\]
Let $\mathcal{B}(\mathcal{P})$ be the subring of $\mathcal{Q}(\mathcal{P})$ generated by cross fractions.

 {\em The swapping fraction (multifraction resp.) algebra of $\mathcal{P}$} is the ring $\mathcal{Q}(\mathcal{P})$ ($\mathcal{B}(\mathcal{P})$ resp.) equipped with the swapping bracket, denoted by $(\mathcal{Q}(\mathcal{P}),\{\cdot, \cdot\})$ ($(\mathcal{B}(\mathcal{P}),\{\cdot, \cdot\})$ resp.).
\end{defn}
By \cite[Proposition 2.9]{Su17}, the ring $\mathcal{B}(\mathcal{P})$ is closed under $\{\cdot,\cdot\}$.

\subsection{Rank $n$ swapping algebra}
\label{rsa}

\begin{defn}{\sc[The rank $n$ swapping ring $\mathcal{Z}_n(\mathcal{P})$]}
Recall the notation in Equation (\ref{equation:notation1}). For $n\geq 2$, let $R_n(\mathcal{P})$ be the ideal of $\mathcal{Z}(\mathcal{P})$ generated by \\ $\left\{ D \in \mathcal{Z}(\mathcal{P}) \; |  \; D = \Delta\left((x_1,\cdots,x_{n+1}), (y_1, \cdots, y_{n+1})\right) ,  \forall x_1, , x_{n+1}, y_1,\cdots, y_{n+1} \in \mathcal{P} \right\}$.

The {\em rank n swapping ring} $\mathcal{Z}_n(\mathcal{P})$ is the quotient ring $\mathcal{Z}(\mathcal{P})/R_n(\mathcal{P})$.
\end{defn}

\begin{figure}
\includegraphics[scale=0.5]{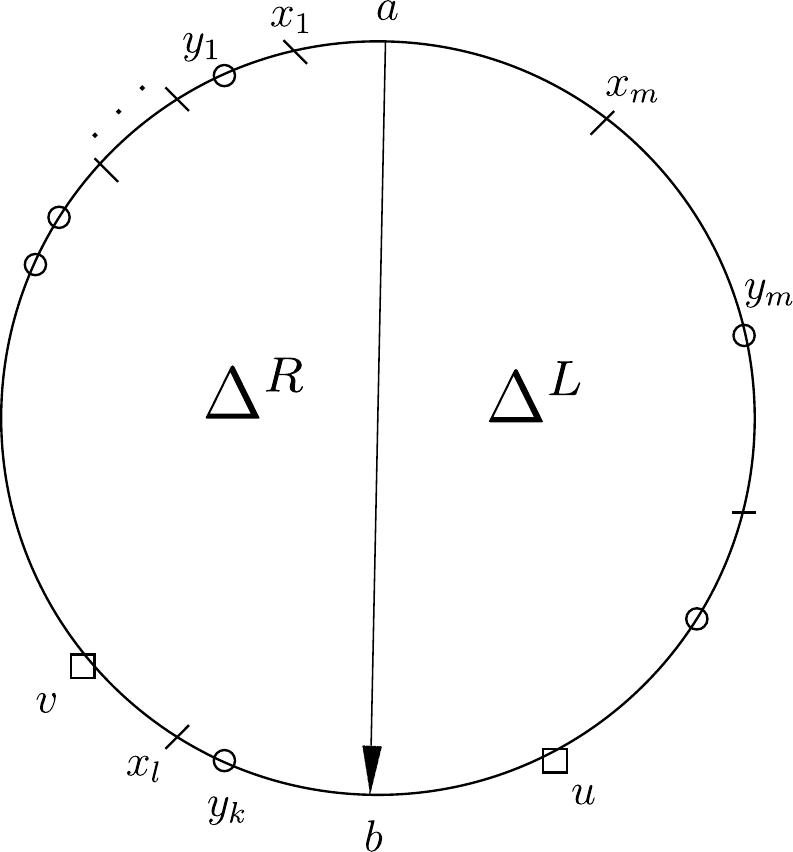}
\caption{$\{ab, \Delta((x_1, \cdots, x_{m}), (y_1,\cdots, y_{m}))\}$}
\label{swapabdet1}
\end{figure}

\begin{lem}{\sc[\cite[Lemma 3.5, Remark 3.6]{Su17}]}
\label{swcal}
For any integer $m \geq 2$, suppose $x_1,\cdots,x_{m}$ $(y_1,\cdots,y_{m} \; resp. )$ in $\mathcal{P}$ are mutually distinct and anticlockwise ordered ($m, x_i, y_i$ used here do not involve $m,x_i, y_i$ in any other places). Assume that $a,b$ belong to $\mathcal{P}$ and $x_{1},\cdots,x_{l}, y_{1},\cdots$, $y_{k}$ are on the \textbf{right} side of the oriented edge $\overrightarrow{ab}$ (include coinciding with $a$ or $b$) as illustrated in Figure \ref{swapabdet1}.
  Let $u$ ($v$ resp.) be strictly on the left (right resp.) side of $\overrightarrow{ab}$. Let
\begin{equation}
\begin{aligned}
\label{equation:R}
&\Delta^R(ab)= \sum_{d=1}^l \mathcal{J}(ab, x_d u)\cdot x_d b \cdot \Delta((x_1, \cdots,x_{d-1}, a, x_{d+1},\cdots,x_{m}), (y_1,\cdots, y_{m}))
  \\&+ \sum_{d=1}^k \mathcal{J}(ab, u y_d)\cdot a y_d \cdot \Delta((x_1, \cdots,x_{m}), (y_1,\cdots,y_{d-1}, b, y_{d+1},\cdots, y_{m})),
\end{aligned}
\end{equation}
\begin{equation}
\begin{aligned}
\label{equation:L}
&\Delta^L(ab)= \sum_{d=k+1}^{m} \mathcal{J}(ab, x_d v)\cdot x_d b \cdot \Delta((x_1, \cdots,x_{d-1}, a, x_{d+1},\cdots,x_{m}), (y_1,\cdots, y_{m}))
  \\&+ \sum_{d=l+1}^{m} \mathcal{J}(ab, v y_d)\cdot a y_d \cdot \Delta((x_1, \cdots,x_{m}), (y_1,\cdots,y_{d-1}, b, y_{d+1},\cdots, y_{m})),
\end{aligned}
\end{equation}

then we have
\[\{ab, \Delta((x_1, \cdots, x_{m}), (y_1,\cdots, y_{m}))\}
   = \Delta^R(ab)=\Delta^L(ab).
\]
\end{lem}

The following proposition is a consequence of the above lemma.
\begin{prop}
The ideal $R_n(\mathcal{P})$ is a Poisson ideal of $\mathcal{Z}(\mathcal{P})$ with respect to the swapping bracket.
\end{prop}

\begin{defn}{\sc[rank $n$ swapping algebra of $\mathcal{P}$]}
The {\em rank $n$ swapping algebra of $\mathcal{P}$} is the ring $\mathcal{Z}_n(\mathcal{P})=\mathcal{Z}(\mathcal{P})/R_n(\mathcal{P})$ equipped with the swapping bracket, denoted by $(\mathcal{Z}_n(\mathcal{P}),\{\cdot, \cdot\})$.
\end{defn}

By Theorem 4.7 in \cite{Su17}, $\mathcal{Z}_n(\mathcal{P})$ is an integral domain. Generators of $\mathcal{Z}_n(\mathcal{P})$ are non-zero divisors, so the cross fraction is well defined in the field of fractions of $\mathcal{Z}_n(\mathcal{P})$.

Let $\mathcal{Q}_n(\mathcal{P})$ be the field of fractions of $\mathcal{Z}_n(\mathcal{P})$. Let $\mathcal{B}_n(\mathcal{P})$ be the sub-fraction ring of $\mathcal{Z}_n(\mathcal{P})$ generated by cross fractions.

\begin{defn}
\label{defnsma}
Then, the {\em rank $n$ swapping fraction (multifraction resp.) algebra of $\mathcal{P}$} is $\mathcal{Q}_n(\mathcal{P})$ ($\mathcal{B}_n(\mathcal{P})$ resp.) equipped with the swapping bracket, denoted by $(\mathcal{Q}_n(\mathcal{P}),\{\cdot, \cdot\})$ ($(\mathcal{B}_n(\mathcal{P}),\{\cdot, \cdot\})$ resp.).
\end{defn}

\section{Fock--Goncharov coordinates}
\label{section:FG}
In this subsection, we explain explicitly the Fock--Goncharov coordinates for $\mathcal{X}_{\operatorname{PGL}_n,\hat{S}}$ in \cite{FG06}. 

Let $\hat{S}= (S,m_b)$ that admits an ideal triangulation, where $S$ is a compact oriented surface and $m_b$ is a finite collection of marked points on $\partial S$ considered modulo isotopy. Let $m_p$ be the set of punctures of $S$. 
An {\em ideal triangulation} $\mathcal{T}$ of $\hat{S}$ is a maximal collection of non-homotopic essential arcs joining points in $m_b \cup m_p$ which are pairwise disjoint on the interior parts.  

\begin{defn}{\sc[$\mathcal{X}_{\operatorname{PGL}_n,\hat{S}}$ \cite[Definition 2.1]{FG06}]}
A {\em $\operatorname{PGL}_n$-framed local system} is a pair $(\rho,\xi)$ where 
\begin{enumerate}
\item $\rho\in \operatorname{Hom}(\pi_1(S),\operatorname{PGL}_n)/\operatorname{PGL}_n$,
\item $\xi$ is a monodromy invariant map from $m_b \cup m_p$ to $\mathcal{B}$.
\end{enumerate}
The moduli space $\mathcal{X}_{\operatorname{PGL}_n,\hat{S}}$ is the collection of equivalent classes of the pairs with the equivalence relation $(\rho,\xi)\sim (g\circ \rho \circ g^{-1},g \circ \xi)$ for any $g \in \operatorname{PGL}_n$.
\end{defn}

\begin{defn}{\em [$n$-TRIANGULATION]}
\label{ntri}
For any triangulation $\mathcal{T}$, we denote its vertices by $V_{\mathcal{T}}$  and its edges by $E_{\mathcal{T}}$.
Given an ideal triangulation $\mathcal{T}$ of $\hat{S}$, we define the {\em $n$-triangulation} $\mathcal{T}_n$ of $\mathcal{T}$ to be: we subdivide each triangle of $\mathcal{T}$ into $n^2$ triangles as shown in Figure \ref{figurekgon}.
\begin{figure}[ht]
\includegraphics[scale=0.3]{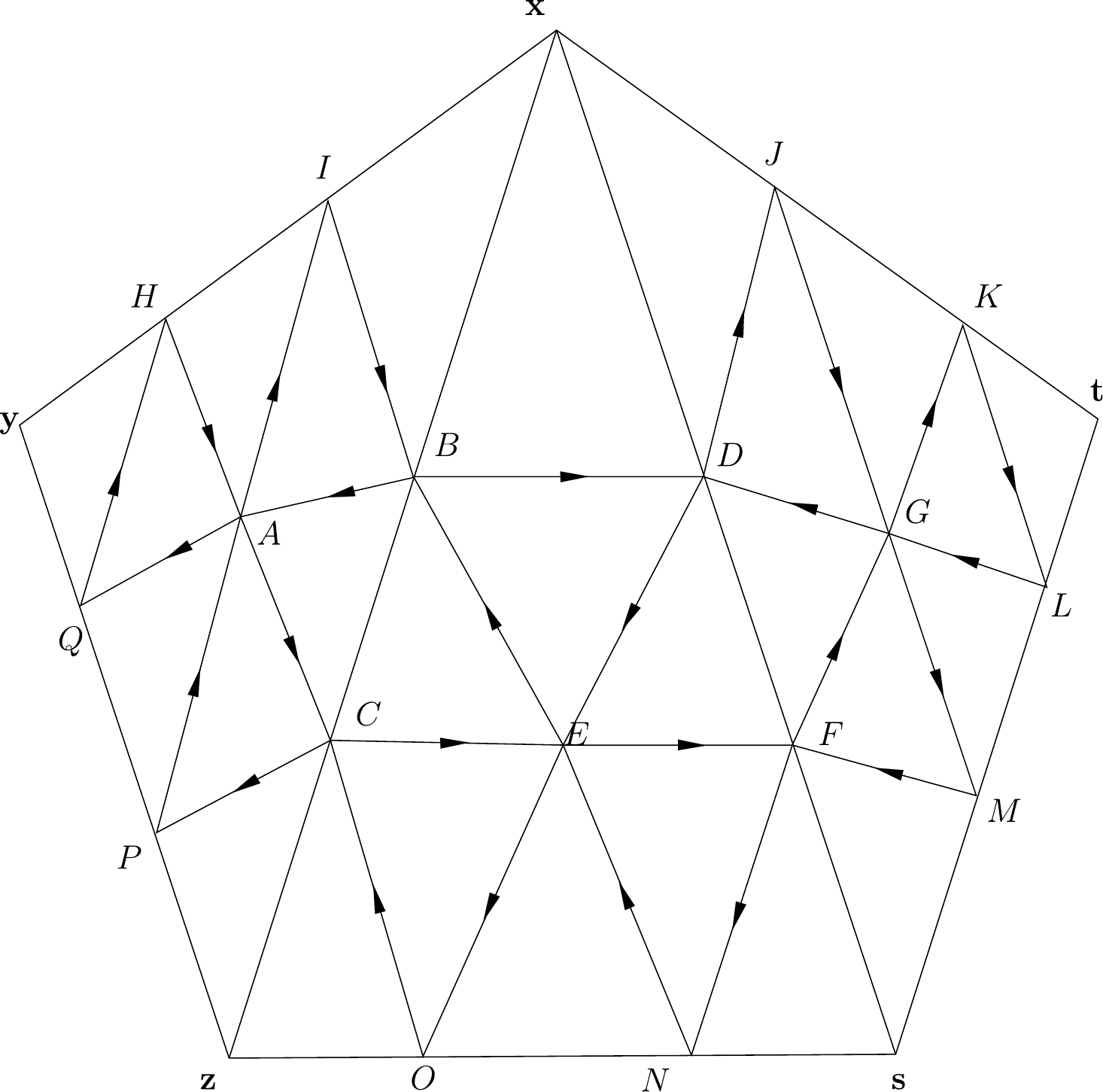}
\caption{$3$-triangulation of $D_5$}
\label{figurekgon}
\end{figure}
Let 
\[
 \mathcal{I}_n= \big(V_{\mathcal{T}_n} \backslash V_{\mathcal{T}}\big) \cap \big(\bigcup_{e\in E_{\mathcal{T}}} e \big),
\]
and
\[
 \mathcal{I}_n'= \left\{V \in \mathcal{I}_n \;|\; V \; is \; not \;on \;the\; boundary\; of \; \mathcal{T}\right\}.
\]
For the case in Figure \ref{figurekgon}, $\mathcal{I}_n' = \{B,C,D,F\}$. Any element $W \in \mathcal{I}_n$ is specified by an oriented edge $\overrightarrow{\textbf{u} \textbf{v}}$ and a pair of positive numbers $(m,l)$ with $m+l = n$ where $m$ is the least number of edges between $W$ and $\textbf{v}$ in $E_{\mathcal{T}_n}$. We also denote $W$ by $v_{\textbf{u}, \textbf{v}}^{m,l}$. For example, the vertex $B=v_{\textbf{x}, \textbf{z}}^{2,1}$ is specified by $\overrightarrow{\textbf{x}\textbf{z}}$ and $(2,1)$.

A {\em marked triangle} $(\textbf{u},\textbf{v},\textbf{w})$ is a triangle $\overline{\textbf{u}\textbf{v}\textbf{w}}$ with a mark on each vertex. 

Let \[\mathcal{J}_n = V_{\mathcal{T}_n} \backslash \big(\mathcal{I}_n \cup V_{\mathcal{T}}\big).\]
For the case in Figure \ref{figurekgon}, $\mathcal{J}_n=\{A,E,G\}$.
Any element $U \in \mathcal{J}_n$ is specified by a marked triangle $(\textbf{u},\textbf{v},\textbf{w})$ and a triple of positive numbers $(m,l,p)$ with $m+l+p =n$ where $m$ ($l$ resp.) is the least number of edges between $U$ and $\overline{\textbf{v}\textbf{w}}$ ($\overline{\textbf{u}\textbf{w}}$ resp.) in $E_{\mathcal{T}_n}$. We also denote $U$ by $v_{\textbf{u},\textbf{v},\textbf{w}}^{m,l,p}$. For example, the vertex $A=v_{\textbf{x},\textbf{y},\textbf{z}}^{1,1,1}$ is specified by $(\textbf{x},\textbf{y},\textbf{z})$ and $(1,1,1)$.

With respect to the orientation of $S$, we define a {\em quiver} $\Gamma_{\mathcal{T}_n}$ with vertices $\mathcal{I}_n\cup \mathcal{J}_n $ and oriented edges as in Figure \ref{figurekgon}.
\end{defn}

\begin{defn}{\sc[Flags]}
Let $E$ be a $n$-dimensional vector space and let $\Omega$ be a volume form of $E$. A {\em flag} is a nested sequence of vector subspaces in $E$
\[F:= F_0\subset F_1 \subset \cdots \subset F_{n-1}\subset F_n=E,\;\; \dim F_i=i\]
equipped with the volume form $\Omega$.

A {\em basis of a flag} $F$ is a basis $\{f_1,\cdots,f_n\}$ of $E$ such that the vectors $f_1,\cdots,f_i$ span the vector space $F_i$ for $i=1,\cdots,n$. 

The {\em flag variety} $\mathcal{B}$ is the collection of the flags.

\end{defn}

\begin{notation}
When $X$ is $\mathcal{T}$ or $\mathcal{T}_n$ or $m_b \cup m_p$ or $\mathcal{I}_n\cup \mathcal{J}_n$, we use $\widetilde{X}$ for denoting all its lifts in the universal cover $\widetilde{S}$. 
\end{notation}

The construction of Fock--Goncharov $\mathcal{X}$ coordinates are based on Lusztig's theory of total positivity \cite{Lu94}\cite{Lu98}.

\begin{defn}{\sc[Fock--Goncharov coordinates \cite[Section 9]{FG06}]}
\label{definition:fgcoor}
Fix an ideal triangulation $\mathcal{T}$ of $\hat{S}$ and its $n$-triangulation $\mathcal{T}_n$.
Given a framed local system $(\rho, \xi)$ of $\mathcal{X}_{\operatorname{PGL}_n,\hat{S}}$, by Deck transformation, we get a $\rho$-equivariant map $\xi_\rho$ from $\widetilde{m_b \cup m_p}$ in the universal cover $\widetilde{S}$ to the flag variety $\mathcal{B}$. For any vertex $\widetilde{V} \in \widetilde{\mathcal{I}_n\cup \mathcal{J}_n}$ of $\widetilde{\mathcal{T}_n}$, suppose that $\widetilde{V}$ is specified by a marked triangle $(\textbf{f},\textbf{g},\textbf{h})$ of $\widetilde{\mathcal{T}}$ and a triple of non-negative integers $(m,l,p)$ with $m+l+p=n$. 
Let
\[\{\textbf{f}_1,\cdots,\textbf{f}_{n}\}, \;\;\{\textbf{g}_1,\cdots,\textbf{g}_{n}\},\;\;\{\textbf{h}_1,\cdots,\textbf{h}_{n}\}\]
be the bases of the three flags $\xi_\rho(\textbf{f})$, $\xi_\rho(\textbf{g})$, $\xi_\rho(\textbf{h})$ respectively. Recall $\Omega$ is the fixed the volume form of $E$. Fix the notation
\[\textbf{v}^i:= \textbf{v}_1 \wedge \cdots \wedge \textbf{v}_i.\] 
We define
$$\Delta_{\widetilde{V}} = \Omega\left(\textbf{f}^{m} \wedge \textbf{g}^{l} \wedge \textbf{h}^{p}\right) .$$ 
Let $\mathcal{FA}_n$ be the fraction ring generated by these determinants over the field $\mathbb{K}$ of characteristic zero with respect to a choice of bases of flags.

For any $V \in \mathcal{I}_n' \cup \mathcal{J}_n$, we choose one of its lift $\widetilde{V}$ in the universal cover $\widetilde{S}$. For all the lifts of the quiver $\Gamma_{\mathcal{T}_n}$ into $\widetilde{S}$, we define $\widetilde{\varepsilon}_{\widetilde{V}\widetilde{W}}$ for any $\widetilde{V},\widetilde{W} \in \widetilde{\mathcal{I}_n\cup \mathcal{J}_n}$ by 
\begin{equation}
\label{equation:epslion}
\widetilde{\varepsilon}_{\widetilde{V}\widetilde{W}} =  \#\{\;arrows \;from\; \widetilde{V} \; to \; \widetilde{W}\;\} -\# \{\;arrows \;from\; \widetilde{W} \; to \; \widetilde{V}\;\}
\end{equation}
as in Figure \ref{figurekgon}. We define
\[X_{V} = \prod_{\widetilde{W}} \Delta_{\widetilde{W}}^{\widetilde{\varepsilon}_{\widetilde{V}\widetilde{W}}},
\]
which does not depend on the lift and the bases of flags that we choose. The collection of $\{X_V\}_{V\in \mathcal{I}'_n\cup \mathcal{J}_n}$ parametrizes $\mathcal{X}_{\operatorname{PGL}_n,\hat{S}}$. 

Let $\mathcal{FX}(\mathcal{T}_n)$ be the fraction ring generated by $\{X_V\}_{V \in \mathcal{I}_n'\cup \mathcal{J}_n}$ over the field $\mathbb{K}$ of characteristic zero. Then $\mathcal{FX}(\mathcal{T}_n) \subset \mathcal{FA}_n$.

\end{defn}

\begin{remark}
\label{remark:FGcoord}
More explicitly, for $V \in \mathcal{J}_n$ corresponding to a marked triangle $(\textbf{f},\textbf{g},\textbf{h})$ of $\widetilde{\mathcal{T}}$ and a triple of positive integers $(m,l,p)$ with $m+l+p=n$, the Fock--Goncharov $\mathcal{X}$ coordinate at $V$, also called the {\em triple ratio}, is 
\[X_V=\frac{\Omega\left(\textbf{f}^{m+1} \wedge \textbf{g}^{l} \wedge \textbf{h}^{p-1}\right)}{\Omega\left(\textbf{f}^{m+1} \wedge \textbf{g}^{l-1} \wedge \textbf{h}^{p}\right)}\cdot \frac{\Omega\left(\textbf{f}^{m-1} \wedge \textbf{g}^{l+1} \wedge \textbf{h}^{p}\right)}{\Omega\left(\textbf{f}^{m} \wedge \textbf{g}^{l+1} \wedge \textbf{h}^{p-1}\right)} \cdot \frac{\Omega\left(\textbf{f}^{m} \wedge \textbf{g}^{l-1} \wedge \textbf{h}^{p+1}\right)}{\Omega\left(\textbf{f}^{m-1} \wedge \textbf{g}^{l} \wedge \textbf{h}^{p+1}\right)}.\]

For $V \in \mathcal{I}'_n$ corresponding to an oriented edge $\overrightarrow{\textbf{x}\textbf{z}}$ with two adjacent anticlockwise oriented ideal triangles $\overrightarrow{\textbf{x}\textbf{y}\textbf{z}}$ and $\overrightarrow{\textbf{x}\textbf{z}\textbf{t}}$, and a pair of positive integers $(m,n-m)$, the Fock--Goncharov $\mathcal{X}$ coordinate at $V$, also called the {\em edge function}, is
\[X_V= \frac{\Omega\left(\textbf{x}^{m} \wedge \textbf{z}^{n-m-1} \wedge \textbf{t}_1 \right)\cdot \Omega\left(\textbf{x}^{m-1} \wedge \textbf{y}_1 \wedge \textbf{z}^{n-m}\right)}{\Omega\left(\textbf{x}^{m}  \wedge \textbf{y}_1 \wedge \textbf{z}^{n-m-1}\right)\cdot \Omega\left(\textbf{x}^{m-1} \wedge \textbf{z}^{n-m} \wedge \textbf{t}_1\right)},
\]
Edge functions generalize Thurston's shear coordinates \cite{T86}.

In both cases, $X_V$ does not depend on the bases that we choose since each term like $\textbf{x}^{m}$ appears once in the numerator, once in the denominator.
Moreover, because $\Omega$ does not change under the projective transformations, $X_V$ is invariant by the projective transformations. So $X_V$ is a well-defined function on $\mathcal{X}_{\operatorname{PGL}_n,\hat{S}}$. 
\end{remark}

\begin{remark}
When the Riemann surface $S$ is closed, with the help of Lie group $G$ invariant functions, Goldman \cite{G86} studied the Hamiltonian flows on $\mathcal{R}_{G,S}$ where the twist flows are described explicitly. For the Hitchin component $H_3(S)$, the flows related to the Fock--Goncharov parameters are studied in \cite{G13} \cite{WZ17}. For the Hitchin component $H_n(S)$, the Hamiltonian flows related to these parameters are studied in \cite{SWZ17} \cite{SZ17}. 

The Fuchsian rigidity with respect to triple ratios (edge functions resp.) can be found in \cite{HS19}. 
\end{remark}

\begin{defn}{\sc[Rank $n$ Fock--Goncharov algebra]}
\label{defn:RnFGA}
Let $\mathcal{FX}(\mathcal{T}_n)$ be the fraction ring generated by $\{X_V\}_{V \in \mathcal{I}_n'\cup \mathcal{J}_n}$ over the field $\mathbb{K}$ of characteristic zero. The {\em rank $n$ Fock--Goncharov Poisson bracket $\{\cdot,\cdot\}_n$} is defined by extending to $\mathcal{FX}(\mathcal{T}_n)$ the following formula for any $V,W \in \mathcal{I}_n' \cup \mathcal{J}_n$ using Leibniz's rule:
\[\left\{X_V ,\; X_W \right\}_{n} =  \varepsilon_{V, W} \cdot X_V \cdot X_W,
\]
where $\varepsilon$ is defined in Equation \eqref{equation:eps}.

The {\em rank $n$ Fock--Goncharov algebra} of $\mathcal{T}_n$ is the ring $\mathcal{FX}(\mathcal{T}_n)$ equipped with the rank $n$ Fock--Goncharov Poisson bracket, denoted by $(\mathcal{FX}(\mathcal{T}_n), \{\cdot, \cdot\}_n)$.

\end{defn}

\begin{remark}
\label{ABGFG}
\begin{enumerate}
\item As shown in \cite[Section 15]{FG06}, the rank $n$ Fock--Goncharov Poisson bracket arises from a special $K_2$ class in $\mathcal{A}_{\operatorname{SL}_n,\hat{S}}$. It can be understood as a canonical Poisson bracket defined for a cluster $\mathcal{X}$ variety \cite{GSV03}.
\item Actually, the rank $n$ Fock--Goncharov Poisson bracket identifies with the Goldman Poisson structure through a different way of symplectic reduction. V. Fock and A. Rosly \cite{FR98} observed that the Goldman Poisson structure on $\mathcal{X}_{G,\hat{S}}$ can be obtained as a quotient of the space of graph connections by the Poisson action of a lattice gauge group endowed with a Poisson-Lie structure. When $G=\operatorname{PGL}(n,\mathbb{R})$, we can calculate Fock--Rosly Poisson bracket between any two Fock--Goncharov coordinates explicitly, which results in the rank $n$ Fock--Goncharov Poisson bracket. In \cite[part I Theorem 3.23]{N13}, Nie uses an approach---the quasi-Poisson structure \cite{AMM98} \cite{AKM02} that is equivalent to that of Fock and Rosly, to explicitly identify the Goldman Poisson structure with the Fock--Goncharov Poisson structure on $\mathcal{X}_{G,\hat{S}}$. 
\end{enumerate}
\end{remark}

\begin{thm}\label{thmcoor}
{\sc[V. V . Fock, A. B. Goncharov \cite[Theorem 1.11]{FG06}, \cite[Theorem 2.5]{FG04} for $n=3$]} Given an ideal triangulation $\mathcal{T}$ and its $n$-triangulation $\mathcal{T}_n$ of $\hat{S}$, the Fock--Goncharov $\mathcal{X}$ coordinates $\{X_V\}_{V \in \mathcal{I}_n'\cup \mathcal{J}_n}$ provide a positive regular atlas on $\mathcal{X}_{\operatorname{PGL}_n,\hat{S}}$.
\end{thm}

\section{$(n\times n)$-determinant ratio}
\label{section:nnratio}
In this section, we construct $(n\times n)$-determinant ratios in $\mathcal{Q}_n(\mathcal{P})$ and relate them with the rank $n$ Fock--Goncharov algebra.

Let us recall the geometric model for $\mathcal{Z}_n(\mathcal{P})$ from \cite[Section 4]{Su17}, which should always be kept in mind while we do the computations in the rank $n$ swapping algebra. Let $\mathcal{P}=\{x_1,\cdots,x_r\}$. We associate a pair $(\mathfrak{a}_i,\mathfrak{b}_i)\in \mathbb{K}^n \times \mathbb{K}^{n*}$ to $x_i$ for $i=1,\cdots,r$. We consider the space $D_{n,r} = (\mathbb{K}^n \times \mathbb{K}^{n*})^r$ of $r$ vectors $\mathfrak{a}_1,\cdots,\mathfrak{a}_r$ in $\mathbb{K}^n$ and $r$ covectors $\mathfrak{b}_1,\cdots,\mathfrak{b}_r$ in $\mathbb{K}^{n*}$. For any $g \in \operatorname{GL}(n,\mathbb{K})$, the action of $g$ on the vector $\mathfrak{a}_i$ is the left multiplication by $g$, the action of $g$ on the covector $\mathfrak{b}_i$ is the right multiplication by $g^{-1}$. We define the product between a vector $\mathfrak{a}_i$ in $\mathbb{K}^n$ and a covector $\mathfrak{b}_j$ in $\mathbb{K}^{n*}$ by $\left<\mathfrak{a}_i|\mathfrak{b}_j\right> := \mathfrak{b}_j(\mathfrak{a}_i)$, which is $\operatorname{GL}(n,\mathbb{K})$ invariant. Let us associate each $\left<\mathfrak{a}_i|\mathfrak{b}_j\right>$ to each ordered pair $x_i x_j \in \mathcal{Z}_n(\mathcal{P})$ as follows. Let $B_{n\mathbb{K}}$ be the subring of $\mathbb{K}[D_{n,r}]$ generated by $\{\left<\mathfrak{a}_i|\mathfrak{b}_j \right>\}_{i=1,j=1}^r$. C. D. Concini and C. Procesi \cite{CP76} proved that
$B_{n\mathbb{K}} = \mathbb{K}[D_{n,r}]^{\operatorname{GL}(n,\mathbb{K})}$. 

Let $W$ be the polynomial ring $\mathbb{K}[\{\mathtt{z}_{i,j}\}_{i,j=1}^r]$,
\[R = \{ f \in W \; |  \; f = \det \left(\begin{array}{lcr}
                                        \mathtt{z}_{i_1 ,j_1} & \cdots & \mathtt{z}_{i_1 ,j_{n+1}} \\
                                        \cdots & \cdots & \cdots \\
                                       \mathtt{z}_{i_{n+1}, j_1} & \cdots & \mathtt{z}_{i_{n+1}, j_{n+1}}
                                      \end{array} \right), \forall i_{k}, j_{l} = 1,\cdots,r \}.\]
Let $T$ be the ideal of $W$ generated by $R$. Then Weyl \cite{W39} show that $B_{n\mathbb{K}} \cong W/T$. Recall $\mathcal{P} = \{x_1,\cdots,x_r\}\subset S^1$. Let $S_{n\mathbb{K}}$ be the ideal of $B_{n\mathbb{K}}$ generated by $\{\left<a_i|b_i\right>\}_{i=1}^r$. Taking quotient by $S_{n\mathbb{K}}$, we identify $\left<\mathfrak{a}_i|\mathfrak{b}_j \right>$ with $x_i x_j$ through $\mathtt{z}_{i,j}$, where we identify $\mathfrak{a}_i$ with $x_i$ on the left and $\mathfrak{b}_j$ with $x_j$ on the right of ordered pair $x_i x_j$ in $\mathcal{Z}_n(\mathcal{P})$.

\begin{defn}
For any $d>1$ and any $x_1,\cdots,x_d, y_1, \cdots, y_d \in \mathcal{P}$, recall the notation
\begin{equation}
\Delta\left((x_1,\cdots,x_d), (y_1, \cdots, y_d)\right): = \det   \left(\begin{array}{cccc}
       x_1 y_1 & \cdots & x_1 y_d \\
       \cdots & \cdots & \cdots \\
       x_d y_1 & \cdots & x_d y_d
     \end{array}\right) \in \mathcal{Z}_n(\mathcal{P}).
\end{equation} 
We call $(x_1,\cdots,x_d)$ ($(y_1,\cdots,y_d)$ resp) the {\em left (right resp.) side $n$-tuple of} the determinant $\Delta\left((x_1,\cdots,x_d), (y_1, \cdots, y_d)\right)$.
\end{defn}
\begin{thm}{\sc[\cite{Su17} Theorem 4.6]}
\label{thm:inv1}
We have $B_{n\mathbb{K}}/{S_{n\mathbb{K}}}  \cong \mathcal{Z}_n(\mathcal{P})$.
\end{thm}

\begin{lem}
For $n>1$, let $x_1,\cdots,x_n, y_1, \cdots, y_n \in \mathcal{P}$. Suppose that $x_1,\cdots,x_n$ ($y_1$, $\cdots$, $y_n$ resp.) are mutually distinct, we have
\[ \Delta\left((x_1,\cdots,x_n), (y_1, \cdots, y_n)\right) \neq 0 .\]
in $\mathcal{Z}_n(\mathcal{P})$.
\end{lem}
\begin{proof}
For any $i=1,\cdots,n$, let $x_i,y_i \in \mathcal{P}$ and let $(x_{i,v}, x_{i,c}) ,(y_{i,v}, y_{i,c})\in \mathbb{K}^n \times \mathbb{K}^{n*}$. Under identification $B_{n\mathbb{K}}/{S_{n\mathbb{K}}}  \cong \mathcal{Z}_n(\mathcal{P})$ of Theorem \ref{thm:inv1}, we identify the vector $x_{i,v}$ in $\mathbb{K}^n$ with $x_i$ in $\mathcal{P}$ on the left and the covector $y_{j,c}$ in $\mathbb{K}^{n*}$ with $y_j$ in $\mathcal{P}$ on the right of ordered pair $x_i y_j$ in $\mathcal{Z}_n(\mathcal{P})$. Then the determinant $\Delta\left((x_1,\cdots,x_n), (y_1, \cdots, y_n)\right)$ is not zero in $\mathcal{Z}_n(\mathcal{P})$ if and only if $\det_{1\leq i,j\leq n}\left(x_{i,v},y_{j,c}\right)$ is not always zero in $\mathbb{K}$ for any generic $\mathbb{K}$-point of $B_{n\mathbb{K}}/{S_{n\mathbb{K}}}$. Actually, for any generic $\mathbb{K}$-point of $B_{n\mathbb{K}}/{S_{n\mathbb{K}}}$, the value of $\det_{1\leq i,j\leq n}\left(x_{i,v},y_{j,c}\right)$ is interpreted as the volume of $x_{1,v},\cdots,x_{n,v}$ with respect to the dual basis of $y_{1,c},\cdots,y_{n,c}$. If $x_{1,v},\cdots,x_{n,v}$ and $y_{1,c},\cdots,y_{n,c}$ are both in general position, the volume $\det_{1\leq i,j\leq n}\left(x_{i,v},y_{j,c}\right)$ is not zero. We conclude that
$$
 \Delta\left((x_1,\cdots ,x_n), (y_1,\cdots,y_n)\right) \neq 0 
$$
in $\mathcal{Z}_n(\mathcal{P})$.
\end{proof}
\begin{prop}
\label{prop Pl}
Let $x_1,\cdots,x_{n-1}, t, y, v_1,\cdots,v_n,u_1 \in \mathcal{P}$. If $x_1,\cdots,x_{n-1},y$ \\ ($v_1,\cdots,v_n,u_1$ resp.) are mutually distinct, we have
\[\frac{\Delta\left((x_1,\cdots, x_{n-1} , t), (v_1,v_2,\cdots, v_n)\right)}{\Delta\left((x_1,\cdots, x_{n-1} , y), (v_1,v_2,\cdots, v_n)\right)} =  \frac{\Delta\left((x_1,\cdots, x_{n-1} , t), (u_1,v_2,\cdots, v_n)\right)}{\Delta\left((x_1,\cdots, x_{n-1} , y), (u_1,v_2,\cdots, v_n)\right)}
\]
in $\mathcal{Q}_n(\mathcal{P})$.
\end{prop}
\begin{proof}
Consider the $(n+1)\times (n+1)$ matrix
\[M = \left(
      \begin{array}{ccccc}
        x_1 u_1 & x_1 v_1 & \cdots & \cdots & x_1 v_n \\
        \cdots & \cdots & \cdots & \cdots & \cdots \\
        x_{n-1} u_1 & x_{n-1} v_1 & \cdots & \cdots & x_{n-1} v_n \\
        t u_1& t v_1 & \cdots & \cdots & t v_n \\
        y u_1& y v_1 & \cdots & \cdots & y v_n \\
      \end{array}
    \right)
\]
The adjugate of $M$ is
\[M^{\star} =  \left(\begin{array}{cccc}
       A_{1,1} & \cdots & (-1)^{n+2}A_{n+1,1} \\
       \cdots & (-1)^{j+i}A_{j,i} & \cdots \\
       (-1)^{n+2}A_{1,n+1}& \cdots & A_{n+1,n+1}
     \end{array}\right)
\]
whose $(i,j)$ entry is $(-1)^{j+i}A_{j,i}$ and $A_{j,i}$ equals the determinant of the $(n \times n)$-matrix obtained from $M$ by deleting the $j$-th row and the $i$-th column. \\
We already know that
$$\det M = 0$$
in $\mathcal{Z}_n(\mathcal{P})$, hence we obtain
\begin{equation}
\label{equA04}
M^{\star} \cdot M = 0_{(n+1)\times (n+1)}.
\end{equation}
The entries of the matrices $M$, $M^{\star}$ and $M^{\star} \cdot M$ are polynomials in $\mathcal{Z}_n(\mathcal{P})$. Recall that by Theorem \ref{thm:inv1}, we identify a vector with a point on the left and a covector with a point on the right of ordered pairs of points in $\mathcal{Z}_n(\mathcal{P})$. When we specify the values of vectors and covectors, under $B_{n\mathbb{K}}/{S_{n\mathbb{K}}}  \cong \mathcal{Z}_n(\mathcal{P})$, we specify the values of all the polynomials in $\mathcal{Z}_n(\mathcal{P})$. Then the values of the matrices $M$, $M^{\star}$ and $M^{\star} \cdot M$ provides the linear endomorphisms of $\mathbb{K}^{n+1}$: $f$, $g$ and $g\circ f$ respectively. Actually, any polynomial $P$ is zero in the ring $\mathcal{Z}_n(\mathcal{P})$ over a field $\mathbb{K}$ of characteristic zero, if and only if $P$ is zero in all of the generic $\mathbb{K}$-points. The following arguments are true for any generic $\mathbb{K}$-point of $B_{n\mathbb{K}}/{S_{n\mathbb{K}}}$, thus true for $\mathcal{Z}_n(\mathcal{P})$. By Equation (\ref{equA04}), the rank of  $g\circ f$ (the dimension of the image of $g\circ f$) is $0$. By the above lemma, we have
\[ \Delta\left((x_1,\cdots ,x_{n-1},y), (v_1,\cdots,v_n)\right) \neq 0 .
\]
Thus, for any generic $\mathbb{K}$-point, the rank of $f$ is at least $n$. Therefore, for any generic $\mathbb{K}$-point, we have the rank of $g$ is at most $1$ (If not so, we will get the rank of $g\circ f$  is not $0$). By considering the top right corner $2\times 2$ minor of $M^{\star}$, for any generic $\mathbb{K}$-point, we have
\[A_{n,1} \cdot A_{n+1,2} - A_{n,2} \cdot A_{n+1,1} = 0,
\]
which implies that 
\[\frac{\Delta\left((x_1,\cdots, x_{n-1} , t), (v_1,v_2,\cdots, v_n)\right)}{\Delta\left((x_1,\cdots, x_{n-1} , y), (v_1,v_2,\cdots, v_n)\right)} =  \frac{\Delta\left((x_1,\cdots, x_{n-1} , t), (u_1,v_2,\cdots, v_n)\right)}{\Delta\left((x_1,\cdots, x_{n-1} , y), (u_1,v_2,\cdots, v_n)\right)}
\]
in $\mathcal{Q}_n(\mathcal{P})$.
\end{proof}

Moreover, by applying Proposition \ref{prop Pl} $n$ times, we have
\begin{cor}
\label{coorep}
Let $x_1,\cdots,x_{n-1}, t, y, v_1,\cdots,v_n,u_1,\cdots,u_n \in \mathcal{P}$. Suppose that $x_1$,$\cdots$, $x_{n-1}$, $y$ ($v_1,\cdots,v_n$ and $u_1,\cdots,u_n$ resp.) are mutually distinct, in $\mathcal{Q}_n(\mathcal{P})$, we have
\[\frac{\Delta\left((x_1,\cdots, x_{n-1} , t), (v_1,\cdots, v_n)\right)}{\Delta\left((x_1,\cdots, x_{n-1} , y), (v_1,\cdots, v_n)\right)} =  \frac{\Delta\left((x_1,\cdots, x_{n-1} , t), (u_1,\cdots, u_n)\right)}{\Delta\left((x_1,\cdots, x_{n-1} , y), (u_1,\cdots, u_n)\right)}.
\]
\end{cor}

By the above corollary, we can define a ratio of two $(n\times n)$-determinants that does not depend on the right side $n$-tuple.

\begin{defn}{\sc[$(n\times n)$-determinant ratio]}
\label{defndr}
Let $x_1,\cdots,x_{n-1},y \in \mathcal{P}$ be different from each other.
The {\em $(n\times n)$-determinant ratio} of $x_1,\cdots,x_{n-1},t,y$:
\[E(x_1,\cdots,x_{n-1}| t,y) := \frac{\Delta\left(\left(x_1,\cdots, x_{n-1} , t\right), \left(v_1,\cdots, v_n\right)\right)}{\Delta\left(\left(x_1, \cdots,x_{n-1} , y\right), \left(v_1,\cdots, v_n\right)\right)}
\]
for any $v_1,\cdots,v_n \in \mathcal{P}$ different from each other. 

The fraction ring $\mathcal{D}_n(\mathcal{P})$ generated by all the $(n \times n)$-determinant ratios is called {\em $(n\times n)$-determinant ratio fraction ring}. 
\end{defn}
\begin{remark}
 The fraction ring $\mathcal{D}_n(\mathcal{P})$ is also a fraction ring generated by all elements of the form $\frac{\Delta\left(\left(x_1,\cdots, x_n\right), \left(v_1,\cdots, v_n\right)\right)}{\Delta\left(\left(y_1, \cdots,y_{n}\right), \left(v_1,\cdots, v_n\right)\right)}$, since
\[
\frac{\Delta\left(\left(x_1,\cdots, x_n\right), \left(v_1,\cdots, v_n\right)\right)}{\Delta\left(\left(y_1, \cdots,y_{n}\right), \left(v_1,\cdots, v_n\right)\right)} = \prod_{i=1}^n E(x_1,\cdots,x_{i-1}, y_{i+1}\cdots, y_{n}| x_{i},y_{i}) .
\]
\end{remark}
By Corollary \ref{coorep}, we have
\begin{cor}
Let $a,b,x_1,\cdots,x_{n-1},t,y\in \mathcal{P}$, $x_1,\cdots,x_{n-1},y$ be different from each other. The value of $$ \{ab, E(x_1,\cdots,x_{n-1}| t,y)\}$$ in $\mathcal{Q}_n(\mathcal{P})$ does not depend on the choice of right side $n$-tuple $(v_1,\cdots,v_n)$.
\end{cor}
As a consequence,
\begin{cor}
\label{corollary:PP}
Let $\mathcal{P}'=\{u_1,\cdots,u_n\} \cup\mathcal{P}$. The value of $$ \left\{ab, \frac{\Delta\left(\left(x_1,\cdots, x_{n-1} , t\right), \left(u_1,\cdots, u_n\right)\right)}{\Delta\left(\left(x_1, \cdots,x_{n-1} , y\right), \left(u_1,\cdots, u_n\right)\right)}\right\}$$ in $\mathcal{Q}_n(\mathcal{P'})$ can be expressed in $\mathcal{Q}_n(\mathcal{P})$ by replacing $u_1,\cdots,u_n$ with any $n$ different elements $v_1,\cdots,v_n$ in $\mathcal{P}$.
\end{cor}
By the above corollary, we can calculate the swapping bracket between two $(n\times n)$-determinant ratios with the right side $n$-tuples in any preferred position.

\section{Main theorem}
\subsection{Homomorphism from rank $n$ Fock--Goncharov algebra to rank $n$ swapping multifraction algebra}
\label{subsection:hom}
Let $D_k$ be a disk $D$ with $k$ points $m_b=\{\textbf{s}\prec \textbf{w}\prec \cdots \prec \textbf{t} \prec \textbf{s}\}$ on $\partial D$, where $\prec$ is defined with respect to the anticlockwise cyclic order on a circle. In this case $\mathcal{X}_{\operatorname{PGL}_n,D_k} \cong \mathcal{B}^k/\operatorname{PGL}_n$ and
\[X_V=\prod_{W \in \mathcal{I}_n \cup \mathcal{J}_n} \Delta_{W}^{\varepsilon_{V,W}}.\]

\begin{defn}
\label{defnhom}
Given an ideal triangulation $\mathcal{T}$ of $D_k$ and its $n$-triangulation $\mathcal{T}_n$, we have the fraction ring $\mathcal{FX}(\mathcal{T}_n)\subset
\mathcal{FA}_n$ as defined in Definition \ref{definition:fgcoor}.

Let 
\[\mathcal{P}= \{s_{n-1}\prec \cdots \prec s_{1} \prec w_{n-1} \prec \cdots \prec w_{1} \prec \cdots \prec t_{n-1} \prec  \cdots \prec t_{1} \prec s_{n-1}\}\]
on $S^1$ with $\#\mathcal{P}=k(n-1)$, where each $\textbf{r}\in m_b$ corresponds to $n-1$ anticlockwise ordered points $r_{n-1},\dots, r_1$ nearby in $\mathcal{P}$ as in Figure \ref{Figure:homo}. 

\begin{figure}
\includegraphics[scale=0.5]{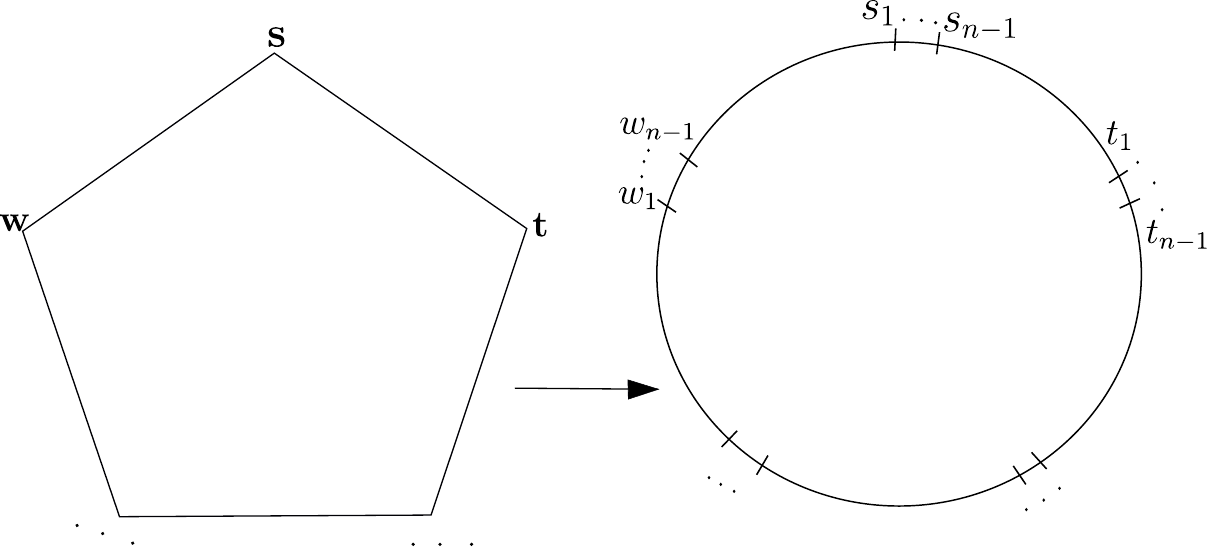}
\caption{$\{\textbf{s}\prec \textbf{w}\prec \cdots \prec \textbf{t} \prec \textbf{s}\}\rightarrow \{s_{n-1}\prec \cdots \prec s_{1} \prec w_{n-1} \prec \cdots \prec w_{1} \prec \cdots \prec t_{n-1} \prec  \cdots \prec t_{1} \prec s_{n-1}\}$}
\label{Figure:homo}
\end{figure}

Fix a choice of distinct $u_1,\cdots, u_n \in \mathcal{P}$, the {\em homomorphism $\chi_n$} (which depends on the choice) from $\mathcal{FA}_n$ to $\mathcal{Q}_n(\mathcal{P})$ is defined by extending the following formula on the generators to $\mathcal{FA}_n$ using Leibniz's rule
\[\chi_n(\Delta_V) = \Delta\left(\left(x_1,\cdots, x_m, y_1,\cdots, y_l, z_1,\cdots, z_p\right), \left(u_1,\cdots u_n\right)\right),\]
where $\Delta_V=\Omega\left(\textbf{x}^{m}\wedge \textbf{y}^{l} \wedge \textbf{z}^{p}\right)$ and any vertex $V$ of $\mathcal{T}_n$ is specified by a marked triangle $(\textbf{x},\textbf{y},\textbf{z})$ of $\mathcal{T}$ and a triple of non-negative integers $(m,l,p)$ with $m+l+p=n$ (Here $\Omega$ is the volume form of $E$ and the bases of the flags $\xi_\rho(\textbf{x})$, $\xi_\rho(\textbf{y})$, $\xi_\rho(\textbf{z})$ are 
\[\{\textbf{x}_1,\cdots,\textbf{x}_{n}\},\;\;\{\textbf{y}_1,\cdots,\textbf{y}_{n}\},\;\;\{\textbf{z}_1,\cdots,\textbf{z}_{n}\}\]
respectively. Recall the notation
$\textbf{v}^i:= \textbf{v}_1 \wedge \cdots \wedge \textbf{v}_i$).

We define the homomorphism $\theta_{\mathcal{T}_n}$ from $\mathcal{FX}(\mathcal{T}_n)$ to $\mathcal{B}_n(\mathcal{P})$ by restricting the homomorphism $\chi_n$ to the fraction ring $\mathcal{FX}(\mathcal{T}_n)$. We have 
\[\theta_{\mathcal{T}_n}(X_V)=\chi_n (X_V) = \prod_W \chi_n(\Delta_W)^{\varepsilon_{VW}}.\] 
\end{defn}
\begin{prop}
The image of $\theta_{\mathcal{T}_n}$ lies in $\mathcal{B}_n(\mathcal{P})$.
\end{prop}
\begin{proof}
For any mutually distinct $v_1,\cdots,v_n\in \mathcal{P}$ ($u_1,\cdots,u_n\in \mathcal{P}$ resp.) and any permutation $\sigma \in S_n$, by \cite{L18} Proposition 2, we have 
\[\frac{v_1 u_{\sigma(1)} \cdots v_n u_{\sigma(n)}}{v_1 u_1 \cdots v_n u_n}  \in \mathcal{B}_n(\mathcal{P}).\]

Thus we obtain
\[\frac{\Delta\big(\big(v_1,\cdots, v_n\big),\big(u_1,\cdots,u_n\big)\big)}{v_1 u_1 \cdots v_n u_n} = \sum_{\sigma \in S_n} \epsilon_\sigma \cdot \frac{v_1 u_{\sigma(1)} \cdots v_n u_{\sigma(n)}}{v_1 u_1 \cdots v_n u_n}\in \mathcal{B}_n(\mathcal{P}).\]

Since $\theta_{\mathcal{T}_n}(X_V)$ can be written as fraction of four or six $\frac{\Delta\big(\big(v_1,\cdots, v_n\big),\big(u_1,\cdots,u_n\big)\big)}{v_1 u_1 \cdots v_n u_n}$'s, we conclude that $\theta_{\mathcal{T}_n}(X_V)$ belongs to $\mathcal{B}_n(\mathcal{P})$.

\end{proof}
\begin{prop}
\label{prop:inj}
The homomorphism $\theta_{\mathcal{T}_n}$ is an injective homomorphism.
\end{prop}  
\begin{proof}
The homomorphism $\chi_n$ sends $(n\times n)$-determinants for $\mathbb{K}^n$ to $(n\times n)$-determinants in $\mathcal{Z}(\mathcal{P})$. By Theorem \ref{thm:inv1}, we have the ring isomorphism $B_{n\mathbb{K}}/{S_{n\mathbb{K}}}  \cong \mathcal{Z}_n(\mathcal{P})$, thus a choice of distinct $u_1,\cdots, u_n \in \mathcal{P}$ for the right side $n$-tuple for $(n\times n)$-determinants in $\mathcal{Z}(\mathcal{P})$ corresponds to fix a basis for $\mathbb{K}^n$. So any relation among $(n\times n)$-determinants for $\mathbb{K}^n$ in $\mathcal{FA}_n$ corresponds to a relation among $(n\times n)$-determinants in $\mathcal{Z}(\mathcal{P})$. Hence it follows that the homomorphism $\chi_n$ is injective. Since the homomorphism $\theta_{\mathcal{T}_n}$ is the homomorphism $\chi_n$ restricted to $\mathcal{FX}(\mathcal{T}_n)$, we conclude that the homomorphism $\theta_{\mathcal{T}_n}$ is injective.
\end{proof}

Recall the notation $w^i:=w_1,\cdots,w_i$.

\begin{prop}
\label{prop:FGswap}
For $V \in \mathcal{J}_n$ associating to $(\textbf{x},\textbf{y},\textbf{z})$ and a triple of positive integers $(m,l,p)$ with $m+l+p=n$, we have
\begin{eqnarray*}
&&\theta_{\mathcal{T}_n}(X_V)
= E\left(x^{m+1},y^{l-1},z^{p-1}|y_l, z_p\right) \cdot
 E\left(x^{m-1},y^{l+1},z^{p-1}|z_p, x_m\right)\cdot
\\&& E\left(x^{m-1},y^{l-1},z^{p+1}|x_m, y_l\right).
\end{eqnarray*}
For $V \in \mathcal{I}_n'$ corresponding to $\overrightarrow{\textbf{x}\textbf{z}}$ and $(i,n-i)$, suppose that two adjacent anticlockwise oriented ideal triangles $\overrightarrow{\textbf{x}\textbf{y}\textbf{z}}$ and $\overrightarrow{\textbf{x}\textbf{z}\textbf{t}}$ have a common edge $\overrightarrow{\textbf{x}\textbf{z}}$, we have
\[\theta_{\mathcal{T}_n}(X_V)=- E\left(x_1,\cdots,x_{i},z_1,\cdots,z_{n-i-1}|t_1,y_1\right)\cdot  E\left(x_1,\cdots,x_{i-1},z_1,\cdots,z_{n-i}|y_1,t_1\right).
\]
\end{prop}
\begin{proof}
We only prove the first case. The other case will follow in a similar way. We have
\begin{eqnarray*}
&&\frac{\Delta\big(\big(x^{m+1}, y^{l}, z^{p-1}\big),\big(u^n\big)\big)}{\Delta\big(\big(x^{m+1}, y^{l-1}, z^{p}\big),\big(u^n\big)\big)}
 = (-1)^{p-1}\frac{\Delta\big(\big(x^{m+1}, y^{l-1}, z^{p-1}, y_l\big),\big(u^n\big)\big)}{\Delta\big(\big(x^{m+1}, y^{l-1}, z^{p-1},z_p\big),\big(u^n\big)\big)} 
\\&=&
 (-1)^{p-1} E\left(x^{m+1},y^{l-1},z^{p-1}|y_l, z_p\right),
\end{eqnarray*}
\begin{eqnarray*}
&&\frac{\Delta\big(\big(x^{m-1}, y^{l+1}, z^{p}\big),\big(u^n\big)\big)}{\Delta\big(\big(x^{m}, y^{l+1}, z^{p-1}\big),\big(u^n\big)\big)}
 = (-1)^{l+p}\frac{\Delta\big(\big(x^{m-1}, y^{l+1}, z^{p-1},z_p\big),\big(u^n\big)\big)}{\Delta\big(\big(x^{m-1}, y^{l+1}, z^{p-1},x_m\big),\big(u^n\big)\big)}
\\&=&
 (-1)^{l+p} E\left(x^{m-1},y^{l+1},z^{p-1}|z_p, x_m\right),
\end{eqnarray*}
\begin{eqnarray*}
&&\frac{\Delta\big(\big(x^{m}, y^{l-1}, z^{p+1}\big),\big(u^n\big)\big)}{\Delta\big(\big(x^{m-1}, y^{l}, z^{p+1}\big),\big(u^n\big)\big)}
 = (-1)^{(l+p)-(p+1)}\frac{\Delta\big(\big(x^{m-1}, y^{l-1}, z^{p+1},x_m\big),\big(u^n\big)\big)}{\Delta\big(\big(x^{m-1}, y^{l-1}, z^{p+1},y_l\big),\big(u^n\big)\big)}
\\&=&
 (-1)^{l-1} E\left(x^{m-1},y^{l-1},z^{p+1}|x_m, y_l\right),
\end{eqnarray*}
Taking the product of the above three terms, we obtain that 
\begin{eqnarray*}
&&\theta_{\mathcal{T}_n}(X_V)
= E\left(x^{m+1},y^{l-1},z^{p-1}|y_l, z_p\right) \cdot
 E\left(x^{m-1},y^{l+1},z^{p-1}|z_p, x_m\right)\cdot
\\&&  E\left(x^{m-1},y^{l-1},z^{p+1}|x_m, y_l\right).
\end{eqnarray*}

\end{proof}

\subsection{Proof of the main theorem}
\label{subsection:mr}
The main technical part of the proof of the main theorem is contained in Proposition \ref{propxyz}. Moreover, in Proposition \ref{propdr} we show how to compute the swapping bracket between two oriented edge ratios. Finally, we give a proof of our main theorem by considering different cases.

\begin{thm}
\label{thmm}
{\sc[Main result]}
\label{mainresult}
Let $D_k$ be a disk with $k$ points on its boundary. For an integer $n>1$, given an ideal triangulation $\mathcal{T}$ of $D_k$ and its $n$-triangulation $\mathcal{T}_n$, the homomorphism $\theta_{\mathcal{T}_n}$ from $\mathcal{FX}(\mathcal{T}_n)$ to $\mathcal{B}_n(\mathcal{P})$ is Poisson with respect to the rank $n$ Fock--Goncharov Poisson bracket and the swapping bracket.
\end{thm}

As shown in Proposition \ref{prop:FGswap}, the image of one Fock--Goncharov $\mathcal{X}$ coordinate can be written as a product of two or three $(n\times n)$-determinant ratios. We start by computing the swapping bracket between two $(n\times n)$-determinants in our cases. Recall the notation in Equation (\ref{equation:notation2}) $[A,B]:=\frac{\{A,B\}}{AB}$.
We will use the following fact frequently, by the Leibniz's rule, $\forall A,B,C,D \in \mathcal{Z}(\mathcal{P})$
\begin{equation}
\label{equ:L}
\left[\frac{A}{B}, \frac{C}{D}\right] = [A, C]-[A, D]-[B, C]+ [B, D].
\end{equation}

\begin{lem}
\label{lemcalMB}
For $n\geq 2$, let $M=\left(c_s d_t\right)_{s,t=1}^n$ be a $(n\times n)$-matrix with $c_s, d_t \in \mathcal{P}$, let $M_{st}$ be the determinant of the matrix obtained from $M$ by deleting the $s$-th row and the $t$-th  column. Let $B \in \mathcal{Q}_n(\mathcal{P})$, we have
\[\left\{\det M, B\right\} = \sum_{s=1}^n \sum_{t=1}^n (-1)^{s+t} \cdot \det M_{st} \cdot \left\{c_s d_t, B\right\}
\]
in $\mathcal{Q}_n(\mathcal{P})$.
\end{lem}
\begin{proof}
Firstly, we have
\[\det M = \sum_{\sigma \in S_n} \epsilon_\sigma \cdot c_1 d_{\sigma(1)} \cdot \cdots \cdot c_n d_{\sigma(n)}
\]
where $S_n$ is the permutation group of $n$ elements, $\epsilon_\sigma$ is the sign of $\sigma$ in $S_n$.
By the Leibniz's rule, we have
\begin{eqnarray*}
&&\left\{\det M, B\right\} = \sum_{\sigma \in S_n} \epsilon_\sigma \cdot \{\prod_{i=1}^n c_i d_{\sigma(i)}, B \}
\\&=& \sum_{\sigma \in S_n} \epsilon_\sigma  \sum_{s=1}^n \prod_{i=1,i\neq s}^n c_i d_{\sigma(i)} \cdot \{c_s d_{\sigma(s)}, B \}
\\&=& \sum_{s=1}^n \sum_{t=1}^n  \left( \sum_{\sigma \in S_n, \sigma(s)=t} \epsilon_\sigma \cdot \prod_{i=1,i\neq s}^n c_i d_{\sigma(i)} \right) \cdot \{c_s d_t, B \}
\\&=&  \sum_{s=1}^n \sum_{t=1}^n (-1)^{s+t} \cdot \det M_{st} \cdot \left\{c_s d_t, B\right\}.
\end{eqnarray*}
We conclude that
\[\left\{\det M, B\right\} = \sum_{s=1}^n \sum_{t=1}^n (-1)^{s+t} \cdot \det M_{st} \cdot \left\{c_s d_t, B\right\}.
\]
\end{proof}
Recall the notation $[A,B]:=\frac{\{A,B\}}{AB}$. Before computing the swapping (Poisson) bracket, Let us recall some useful properties:
for any $A,B\in \mathcal{Q}_n(\mathcal{P})$
\[[A,B]=[-A,B]=[A,-B]=[-A,-B],\]
\[[1/A,B]=-[A,B],\;\;\; [A,B]=-[B,A],\]
\[ [A,A]=0.\]
Recall the notation $w^i:=w_1,\cdots,w_i$.
\begin{prop}{\sc[Main proposition]}
\label{propxyz}
Suppose that $x_{n-1}\prec \cdots \prec x_1 \prec v_1\prec\cdots \prec v_n\prec y_{n-1}\prec\cdots \prec y_1 \prec z_{n-1}\prec\cdots \prec z_1\prec u_1\prec\cdots \prec u_n \prec x_{n-1}$ are ordered anticlockwise in $\mathcal{P}$ as shown in Figure \ref{Figure:swapother4}, for non-negative integers $m,l,p,m',l',p'$ with $m+l+p=n$ and $m'+l'+p'=n$. 

\begin{figure}
\includegraphics[scale=0.5]{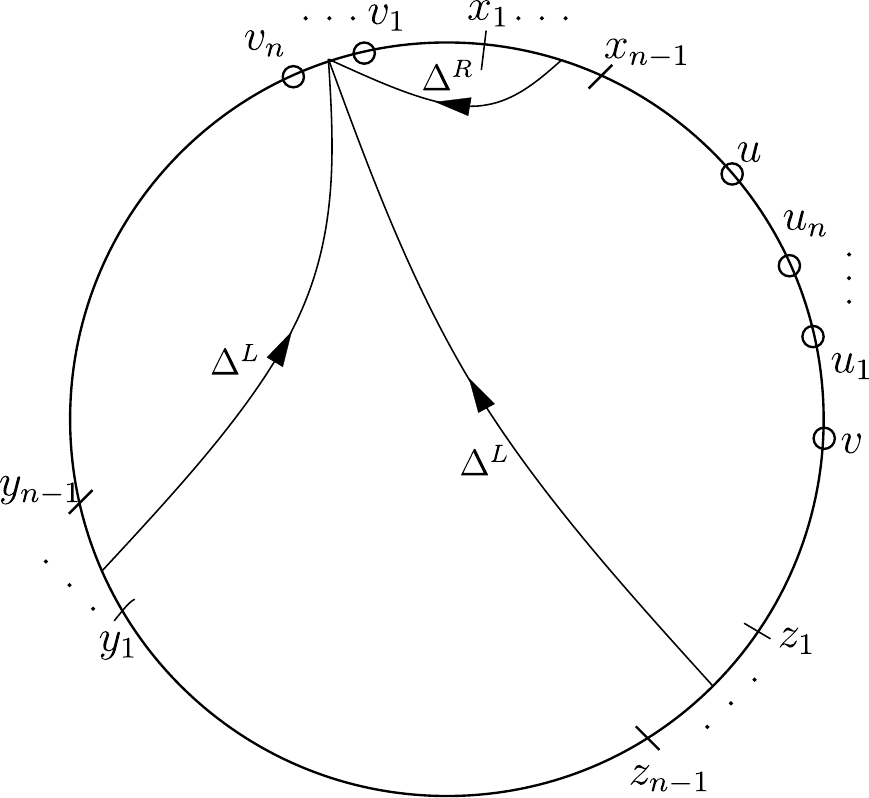}
\caption{Computation of $\left[\Delta\left(\left(x^m, y^l, z^p \right), \left(v^n\right)\right), \Delta\left(\left(x^{m'}, y^{l'}, z^{p'}\right), \left(u^n\right)\right)\right]$}
\label{Figure:swapother4}
\end{figure}
\textbf{If $l\geq l'$ or $p \leq p'$}(*) as in Figure \ref{Figure:propcond}, we have
\begin{eqnarray*}
&&C= \left[\Delta\left(\left(x^m, y^l, z^p \right), \left(v^n\right)\right), \Delta\left(\left(x^{m'}, y^{l'}, z^{p'}\right), \left(u^n\right)\right)\right]
\\&=&\frac{1}{2}\cdot \min\{m,m'\}- \frac{1}{2}\cdot \min\{l, l'\} - \frac{1}{2}\cdot \min\{p, p'\}
\end{eqnarray*}
in $\mathcal{Q}_n(\mathcal{P})$.
\end{prop}
\begin{remark}
Lemma \ref{swcal} and Lemma \ref{lemcalMB} allows us to compute the swapping bracket between any two $(n\times n)$-determinants. The general result is complicated. But with respect to the cyclic order in Figure \ref{Figure:swapother4}, the formula of Proposition \ref{propxyz} is simple under the condition (*) $l\geq l'$ or $p \leq p'$.
Essentially, the $+$ and $-$ sign before $\frac{1}{2}\cdot \min$ in our formula is due to our cyclic order.

The condition (*) is strict and is used in case $3$ of the proof. This condition depends on the cyclic order of points and is crucial to the proof of the main theorem. Finding the proper cyclic order and the condition (*) for this proposition is not as direct as the proof of this proposition.
\end{remark}

\begin{figure}
\includegraphics[scale=0.5]{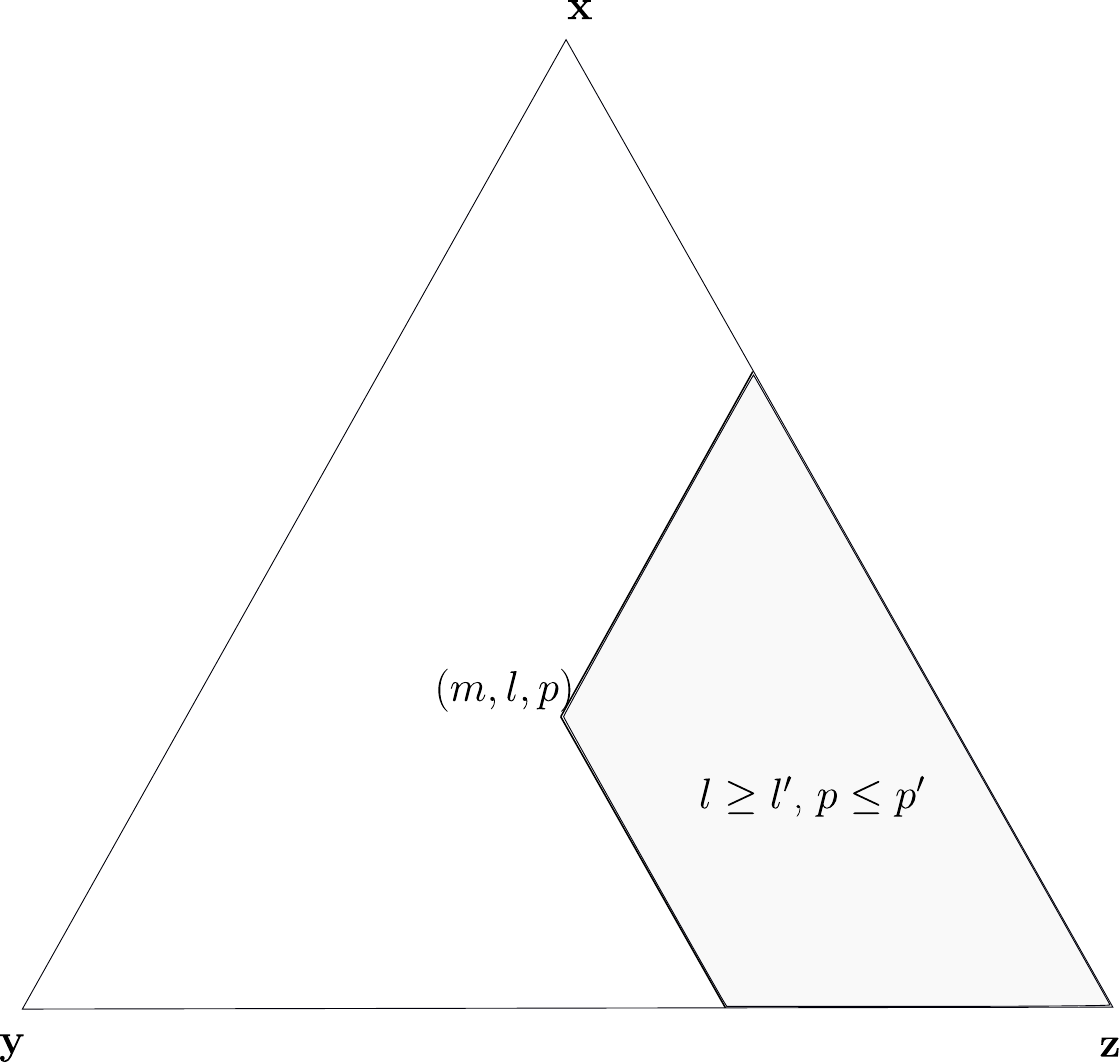}
\caption{The region of $(m',l',p')$ satisfies the condition (*).}
\label{Figure:propcond}
\end{figure}

\begin{proof}
Let $M=(c_s v_t)_{s,t=1}^n$ be a $(n\times n)$ matrix with $c_s=x_s$ for $s=1,\cdots,m$, $c_s=y_{s-m}$ for $s=m+1,\cdots, m+l$ and $c_s=z_{s-m-l}$ for $s=m+l+1,\cdots, n$. Then $\det M = \Delta\left(\left(x^m, y^l, z^p \right), \left(v^n\right)\right)$. Let $B = \Delta\left(\left(x^{m'}, y^{l'}, z^{p'}\right), \left(u^n\right)\right)$.

By Lemma \ref{lemcalMB}, we have 
\[C=\sum_{s=1}^n \frac{1}{\det M \cdot B} \sum_{t=1}^n (-1)^{s+t} \cdot M_{st} \cdot \left\{c_s v_t, B\right\}.
\] 
Given $s=1,\cdots,n$, we compute the sum 
\[\frac{1}{\det M \cdot B}\sum_{t=1}^n (-1)^{s+t} M_{st} \cdot \left\{c_s v_t, B\right\}\]
over $t$, where the summation is called the {\em sum over $t$ for $c_s$} for short. Thus $C$ equals to the sum of the above term over $s$. In the following three cases $x_s v_t$ or $y_{s-m} v_t$ or $z_{s-m-l} v_t$ takes the place of $c_s v_t$.

\textbf{1. For the sum over $t$ for $x_s$ where $1 \leq s \leq m$:}

Let us fix the notation 
\[w^i\backslash w_j \cup x := w_1,\cdots, w_{j-1},x,w_{j+1},\cdots,w_i.\]
Here $u \in S^1$ is strictly on the left side of $\overrightarrow{x_s v_t}$ for any possible $s$ and $t$. By Lemma \ref{swcal} Equation (\ref{equation:R}), and using $\Delta^R(x_s v_t)$ with respect to the right side of $\overrightarrow{x_s v_t}$, we have
\[ \left\{x_s v_t,  \Delta\left(\left(x^{m'}, y^{l'}, z^{p'}\right), \left(u^n\right)\right)\right\}=\Delta^R(x_s v_t)=\sum_{i=1}^s Q_{x_i},\]
where $Q_{x_i}=\mathcal{J}(x_s v_t, x_i u)\cdot x_i v_t \cdot \Delta\left(\left(x^{m'}\backslash x_i \cup x_s , y^{l'}, z^{p'}\right), \left(u^n\right)\right)$.
\begin{enumerate}
  \item Suppose that $1\leq s\leq m'$. If $i \leq m'$ and $i \neq s$, then $x_s$ appears twice in the left side $n$-tuple $\left(x^{m'}\backslash x_i \cup x_s, y^{l'}, z^{p'}\right)$. We obtain
\[Q_{x_i}=\begin{cases}

   \frac{1}{2} \cdot x_s v_t \cdot \Delta\left(\left(x^{m'}, y^{l'}, z^{p'}\right), \left(u^n\right)\right) &\mbox{if $i=s$,}\\

   0 &\mbox{if $i<s$.}
\end{cases} \] 
Thus we get
\[\left\{x_s v_t,  \Delta\left(\left(x^{m'}, y^{l'}, z^{p'}\right), \left(u^n\right)\right)\right\}=\frac{1}{2} \cdot x_s v_t \cdot \Delta\left(\left(x^{m'}, y^{l'}, z^{p'}\right), \left(u^n\right)\right).
\]
Hence for $1\leq s\leq \min\{m,m'\}$ the sum over $t$ for $x_s$ equals \[\frac{1}{\det M \cdot B}\sum_{t=1}^n (-1)^{s+t} \cdot \det M_{st} \cdot \left\{c_s v_t, B\right\} = \frac{1}{\det M \cdot B}\sum_{t=1}^n (-1)^{s+t} M_{st} \cdot \frac{1}{2}\cdot c_s v_t \cdot B  =\frac{1}{2}.\]
  \item Suppose that $m'< s \leq m$, we have
\begin{eqnarray*}
&&\left\{x_s v_t,  \Delta\left(\left(x^{m'}, y^{l'}, z^{p'}\right), \left(u^n\right)\right)\right\}
\\&=&\sum_{i=1}^{m'} x_i v_t \cdot \Delta\left(\left(x^{m'}\backslash x_i \cup x_s, y^{l'}, z^{p'}\right), \left(u^n\right)\right).
\end{eqnarray*}
The sum over $t$ for $x_s$ equals
\[\frac{1}{\det M \cdot B}\sum_{i=1}^{m'} \left( \Delta((x^{m}\backslash x_s \cup x_i, y^l, z^p ), (v^n)) \cdot
  \Delta((x^{m'}\backslash x_i \cup x_s, y^{l'}, z^{p'}), (u^n)) \right).
\]

Since $i \leq m'<s \leq m$, we have $x_i$ appears twice in the left side $n$-tuple $\left(x^{m}\backslash x_s \cup x_i, y^l, z^p \right)$. So
$$\Delta\left(\left(x^{m}\backslash x_s \cup x_i, y^l, z^p \right), \left(v^n\right)\right) =0.$$
Hence in this case the sum over $t$ for $x_s$ equals $0$.
\end{enumerate}
The computations of the other two cases follow the similar strategy as above.

\textbf{2. For the sum over $t$ for $y_s$ where $1 \leq s \leq l$}:

Here $v \in S^1$ is strictly on the right side of $\overrightarrow{y_s v_t}$ for any possible $s$ and $t$. By Lemma \ref{swcal} Equation (\ref{equation:L}), and using $\Delta^L(y_s v_t)$ with respect to the left side of $\overrightarrow{y_s v_t}$, we have 
\[ \left\{y_s v_t,  \Delta\left(\left(x^{m'}, y^{l'}, z^{p'}\right), \left(u^n\right)\right)\right\}=\sum_{i=s}^{l'} Q_{y_i},\]
where $Q_{y_i}=\mathcal{J}(y_s v_t, y_i v)\cdot y_i v_t \cdot \Delta\left(\left(x^{m'}, y^{l'}\backslash y_i \cup y_s, z^{p'}\right), \left(u^n\right)\right)$.
\begin{enumerate}
  \item Suppose that $1 \leq s\leq l'$. If $i \leq l'$ and $i\neq s$, then $y_s$ appears twice in the left side $n$-tuple $\left(x^{m'}, y^{l'}\backslash y_i \cup y_s, z^{p'}\right)$. We have
\[Q_{y_i}=\begin{cases}

   -\frac{1}{2} \cdot y_s v_t \cdot \Delta\left(\left(x^{m'}, y^{l'}, z^{p'}\right), \left(u^n\right)\right) &\mbox{if $i=s$,}\\

   0 &\mbox{if $s< i \leq l'$.}
\end{cases} \] 
Hence for $1\leq s \leq \min\{l,l'\}$ the sum over $t$ for $y_s$ equals $-\frac{1}{2}$.
  \item Suppose that $l'< s \leq l$, the summation is null in this case. Thus we obtain 
\[ \left\{y_s v_t,  \Delta\left(\left(x^{m'}, y^{l'}, z^{p'}\right), \left(u^n\right)\right)\right\}=0.\]
Hence in this case the sum over $t$ for $y_s$ equals $0$.
\end{enumerate}

\textbf{3. For the sum over $t$ for $z_s$ where $1 \leq s \leq p$:
}

Here $v \in S^1$ is strictly on the right side of $\overrightarrow{z_s v_t}$ for any possible $s$ and $t$. By Lemma \ref{swcal} Equation (\ref{equation:L}), and using $\Delta^L(z_s v_t)$ with respect to the left side of $\overrightarrow{z_s v_t}$, we have 
\begin{equation}
\begin{aligned}
\label{equation:zv}
&\left\{z_s v_t,  \Delta\left(\left(x^{m'}, y^{l'}, z^{p'}\right), \left(u^n\right)\right)\right\}
\\&=\sum_{i=s}^{p'}\mathcal{J}(z_s v_t, z_i v)\cdot z_i v_t \cdot \Delta\left(\left(x^{m'}, y^{l'}, z^{p'}\backslash z_i \cup z_s\right), \left(u^n\right)\right)
\\&+\sum_{i=1}^{l'} \mathcal{J}(z_s v_t, y_i v)\cdot y_i v_t \cdot \Delta\left(\left(x^{m'}, y^{l'}\backslash y_i \cup z_s, z^{p'}\right), \left(u^n\right)\right).
\end{aligned}
\end{equation}

\begin{enumerate}
\item Suppose that $1\leq s\leq p'$. If $i\leq p'$ and $i\neq s$, we have $z_s$ appears twice in the left side $n$-tuple $\left(x^{m'}, y^{l'}, z^{p'}\backslash z_i \cup z_s\right)$. Thus we get
\begin{eqnarray*}
&&\mathcal{J}(z_s v_t, z_i v)\cdot z_i v_t \cdot \Delta\left(\left(x^{m'}, y^{l'}, z^{p'}\backslash z_i \cup z_s\right), \left(u^n\right)\right)
\\&=&\begin{cases}

   -\frac{1}{2} \cdot z_s v_t \cdot \Delta\left(\left(x^{m'}, y^{l'}, z^{p'}\right), \left(u^n\right)\right) &\mbox{if $i=s$,}\\

   0 &\mbox{if $s < i\leq p'$.}
\end{cases} 
\end{eqnarray*}
For the third line of Equation \eqref{equation:zv}, since $z_s$ appears twice in the left side $n$-tuple $\left(x^{m'}, y^{l'}\backslash y_i \cup z_s, z^{p'}\right)$, we obtain
\[\Delta\left(\left(x^{m'}, y^{l'}\backslash y_i \cup z_s, z^{p'}\right), \left(u^n\right)\right)=0.\]
Thus we get 
\begin{eqnarray*}
&&\left\{z_s v_t,  \Delta\left(\left(x^{m'}, y^{l'}, z^{p'}\right), \left(u^n\right)\right)\right\}
= -\frac{1}{2} \cdot z_s v_t \cdot  \Delta\left(\left(x^{m'}, y^{l'}, z^{p'}\right), \left(u^n\right)\right).
\end{eqnarray*}
Hence for $1\leq s \leq \min\{p,p'\}$ the sum over $t$ for $z_s$ equals $-\frac{1}{2}$.
\item Suppose that $p'< s\leq p$, \textbf{by our condition} (*), we have $l\geq l'$. We get 
\begin{eqnarray*}
&&\left\{z_s v_t,  \Delta\left(\left(x^{m'}, y^{l'}, z^{p'}\right), \left(u^n\right)\right)\right\}
\\&=&\sum_{i=1}^{l'} \mathcal{J}(z_s v_t, y_i v)\cdot y_i v_t \cdot \Delta\left(\left(x^{m'}, y^{l'}\backslash y_i \cup z_s, z^{p'}\right), \left(u^n\right)\right)
\\&=&-\sum_{i=1}^{l'}  y_i v_t \cdot \Delta\left(\left(x^{m'}, y^{l'}\backslash y_i \cup z_s, z^{p'}\right), \left(u^n\right)\right).
\end{eqnarray*}
Thus the sum over $t$ for $z_s$ equals

\[\frac{-1}{\det M \cdot B}\sum_{i=1}^{l'} \Delta\left(\left(x^m, y^l, z^{p}\backslash z_s \cup y_i \right), (v^n)\right) \cdot
 \Delta\left(\left(x^{m'}, y^{l'}\backslash y_i \cup z_s , z^{p'}\right), (u^n)\right).\]
Observe that $i \leq l' \leq l$, thus $y_i$ appears twice in the left side $n$-tuple $\left(x^m, y^l, z^{p}\backslash z_s \cup y_i \right)$, hence $\Delta\left(\left(x^m, y^l, z^{p}\backslash z_s \cup y_i \right), (v^n)\right)=0$. 
Thus the above summation is zero. Hence in this case the sum over $t$ for $z_s$ equals $0$.
\end{enumerate}
Sum over all the above cases, we conclude that
\[C =  \frac{1}{2}\cdot \min\{m,m'\}- \frac{1}{2}\cdot \min\{l, l'\} - \frac{1}{2}\cdot \min\{p, p'\}.
\]
\end{proof}

\begin{defn}{\sc[oriented edge ratio]}
Given an ideal triangulation $\mathcal{T}$ and its $n$-triangulation $\mathcal{T}_n$ of $D_k$. For the vertex $V \in \mathcal{I}_n \cup \mathcal{J}_n$ associating to the marked triangle $(\textbf{x},\textbf{y},\textbf{z})$ and a triple of non-negative integers $(m,l,p)$ with $m+l+p=n$, we express the vertex $V$ by $v_{\textbf{x},\textbf{y},\textbf{z}}^{m,l,p}$. To the {\em oriented edges}
\[e_1=\overrightarrow{v_{\textbf{x},\textbf{y},\textbf{z}}^{m+1,l,p-1} v_{\textbf{x},\textbf{y},\textbf{z}}^{m+1,l-1,p}}, \;e_2=\overrightarrow{v_{\textbf{x},\textbf{y},\textbf{z}}^{m-1,l+1,p} v_{\textbf{x},\textbf{y},\textbf{z}}^{m,l+1,p-1}},\;e_3=\overrightarrow{v_{\textbf{x},\textbf{y},\textbf{z}}^{m,l-1,p+1} v_{\textbf{x},\textbf{y},\textbf{z}}^{m-1,l,p+1}}\] of $\mathcal{T}_n$ without touching the vertices of $\mathcal{T}$, we associate 
\[E_{e_1}:=E\left(x^{m+1},y^{l-1},z^{p-1}|y_{l},z_{p}\right),\]
\[E_{e_2}:=E\left(x^{m-1},y^{l+1},z^{p-1}|z_{p},x_{m}\right),\]
\[E_{e_3}:=E\left(x^{m-1},y^{l-1},z^{p+1}|x_{m},y_{l}\right).\]
 Each one of them is called {\em oriented edge ratio} of the corresponding arrow. 
\end{defn}
\begin{lem}
\label{lemct}
The image of the Fock--Goncharov coordinates by $\theta_{\mathcal{T}_n}$ are the products of direct edge ratios as in Figure \ref{imte}. 
\end{lem}
\begin{figure}
\includegraphics[scale=0.5]{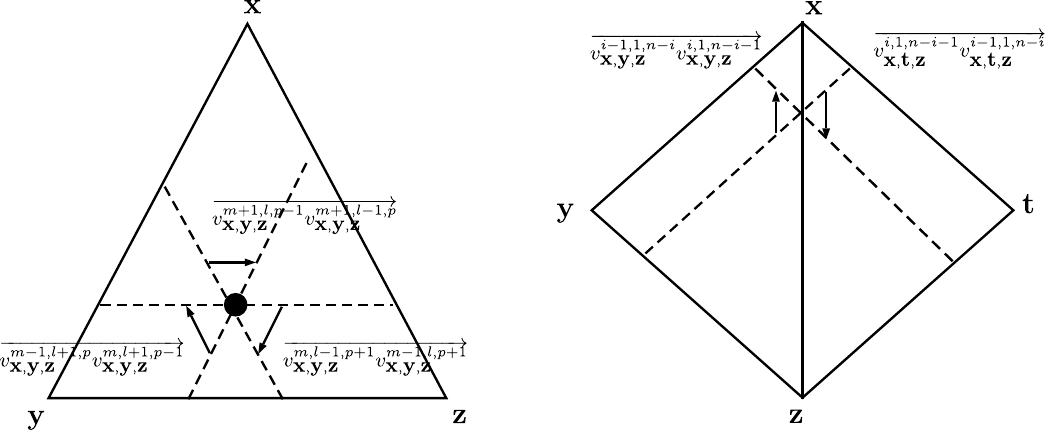}
\caption{$\theta_{\mathcal{T}_n}(T_{m,l,p}(X,Y,Z))$ and $\theta_{\mathcal{T}_n} (\mathbb{B}_i(Y,T,Z,X))$}
\label{imte}
\end{figure}
\begin{proof}
By Proposition \ref{prop:FGswap}, when $V \in \mathcal{J}_n$ is specified by $(\textbf{x},\textbf{y},\textbf{z})$ and a triple of positive integers $(m,l,p)$ with $m+l+p=n$, we have
\[
\theta_{\mathcal{T}_n}(X_V) = 
E_{\overrightarrow{v_{\textbf{x},\textbf{y},\textbf{z}}^{m+1,l,p-1} v_{\textbf{x},\textbf{y},\textbf{z}}^{m+1,l-1,p}}} \cdot
E_{\overrightarrow{v_{\textbf{x},\textbf{y},\textbf{z}}^{m-1,l+1,p} v_{\textbf{x},\textbf{y},\textbf{z}}^{m,l+1,p-1}}} \cdot
E_{\overrightarrow{v_{\textbf{x},\textbf{y},\textbf{z}}^{m,l-1,p+1} v_{\textbf{x},\textbf{y},\textbf{z}}^{m-1,l,p+1}}}.
\]

For $V \in \mathcal{I}_n'$ associating to $\overrightarrow{\textbf{x}\textbf{z}}$ and $(i,n-i)$, suppose that two adjacent anticlockwise oriented ideal triangles $\overrightarrow{\textbf{x}\textbf{y}\textbf{z}}$ and $\overrightarrow{\textbf{x}\textbf{z}\textbf{t}}$ have a common edge $\overrightarrow{\textbf{x}\textbf{z}}$, we have
 
\[
\theta_{\mathcal{T}_n} (X_V)
=- E_{\overrightarrow{v_{\textbf{x},\textbf{y},\textbf{z}}^{i-1,1,n-i} v_{\textbf{x},\textbf{y},\textbf{z}}^{i,1,n-i-1}}} \cdot E_{\overrightarrow{v_{\textbf{x},\textbf{t},\textbf{z}}^{i,1,n-i-1} v_{\textbf{x},\textbf{t},\textbf{z}}^{i-1,1,n-i}}}.
\]
\end{proof}
\begin{defn}{\sc[level]}
Given a triangulation $\mathcal{T}$ and its $n$-triangulation $\mathcal{T}_n$. For any vertex $\textbf{x}$ of $\mathcal{T}$, we define the {\em $k$-th level} of $\textbf{x}$ to be the union of the edges in 
\[\{v_{\textbf{x},\textbf{y},\textbf{z}}^{n-k,l,k-l} v_{\textbf{x},\textbf{y},\textbf{z}}^{n-k,l-1,k-l+1}\;|\;\overline{\textbf{x}\textbf{y}\textbf{z}} \;is \; a\; triangle\; of \; \mathcal{T}\; for \; some\; \textbf{y},\textbf{z},\; l=1,\cdots,k\}.\]
For any oriented edge $e$ lying between $i$-th level and $(i+1)$-th level of $\textbf{x}$, the {\em sign $\epsilon_{\textbf{x}}(e)$ of $e$ with respect to} $\textbf{x}$ is $+1$ ($-1$ resp.) if the arrow of $e$ goes from the $(i+1)$-th level of $\textbf{x}$ to the $i$-th level of $\textbf{x}$ (otherwise resp.). For example $\epsilon_{\textbf{x}}(e_1)=\epsilon_{\textbf{x}}(e_2)=1$ in Figure \ref{der1}(1)(2).

When the triangulation $\mathcal{T}$ is an ideal triangulation of $D_k$, the $k$-th level of $\textbf{x}$ is topologically an interval, not a circle.
In this case, for any two oriented edges $e_1$ and $e_2$ in $\mathcal{T}_n$, we say $e_2$ is {\em after (before resp.) $e_1$ for} $\textbf{x}$, if $e_1$ and $e_2$ both lie between $i$-th level and $(i+1)$-th level of $\textbf{x}$ and $e_2$ is strictly after (before resp.) $e_1$ with respect to anticlockwise orientation centered at $\textbf{x}$ as in Figure \ref{der1}(1) ((2) resp.). 
\end{defn}

\begin{prop}
\label{propdr}
Given an ideal triangulation $\mathcal{T}$ of $D_k$ and its $n$-triangulation $\mathcal{T}_n$.
 Let $e_1$ ($e_2$ resp.) be the oriented edge of $\mathcal{T}_n$ lying inside the ideal triangle $\overline{\textbf{x}\textbf{y}\textbf{z}}$ ($\overline{\textbf{x}'\textbf{y}'\textbf{z}'}$ resp.) of $\mathcal{T}$. Then as in Figure \ref{der1} and \ref{der2}, we have
\[\left[E_{e_1}, E_{e_2}\right]
=\begin{cases}\frac{1}{2} \cdot \epsilon_{\textbf{u}}(e_1) \cdot \epsilon_{\textbf{u}}(e_2) ,& e_2 \;\;is\;\;after\;\; e_1 \;\; for \;\; \textbf{u} \in \{\textbf{x},\textbf{y},\textbf{z}\}\cap \{\textbf{x}',\textbf{y}',\textbf{z}'\}; \cr -\frac{1}{2}\cdot \epsilon_{\textbf{u}}(e_1) \cdot \epsilon_{\textbf{u}}(e_2) ,&e_2\;\;is\;\; before\;\; e_1 \;\; for \;\; \textbf{u}\in \{\textbf{x},\textbf{y},\textbf{z}\}\cap \{\textbf{x}',\textbf{y}',\textbf{z}'\};\cr 0,&otherwise.\end{cases}
\]
\end{prop}

\begin{figure}
\includegraphics[scale=0.4]{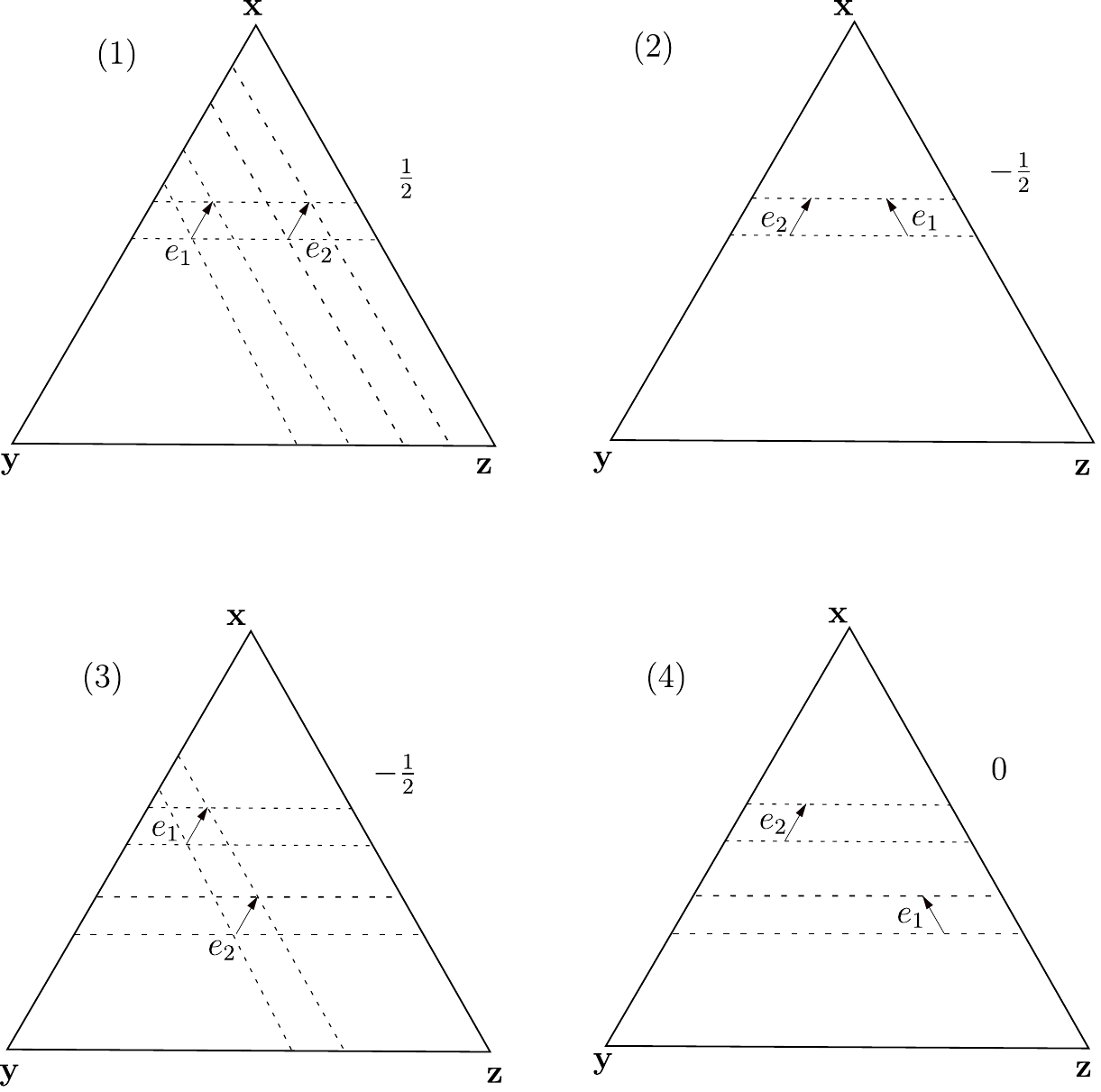}
\caption{Computation of $[E_{e_1},E_{e_2}]$. Consider the number of vertices that each vertex has $e_1$ and $e_2$ lying between two successive levels of that vertex. For (1)(3) the number is two, say $\textbf{x}$ and $\textbf{y}$; for (2)(4) the number is one, say $\textbf{x}$. }
\label{der1}
\end{figure}

\begin{proof}
The number of elements in $\#\left(\{\textbf{x},\textbf{y},\textbf{z}\}\cap \{\textbf{x'},\textbf{y'},\textbf{z'}\}\right)$, denoted by $\mathcal{N}$, is zero or one or two or three. When $\mathcal{N}=0$, we have $\left[E_{e_1}, E_{e_2}\right]=0$ since all the linking numbers are zero by setting the right side $n$-tuples properly.

When $\mathcal{N}=3$, suppose that the triangle $\overline{\textbf{x}\textbf{y}\textbf{z}}$ of $\mathcal{T}$ has $\textbf{x}\prec \textbf{y} \prec \textbf{z} \prec \textbf{x}$ with respect to anticlockwise orientation of $\partial D$. Suppose that $x_{n-1}\prec \cdots \prec x_1 \prec y_{n-1}\prec\cdots \prec y_1 \prec z_{n-1}\prec\cdots \prec z_1\prec  x_{n-1}$ are anticlockwise ordered in $\mathcal{P}$ as shown in Figure \ref{Figure:swapother4}.
In the triangle $\overline{\textbf{x}\textbf{y}\textbf{z}}$, for any given oriented edge, there are two different vertices in $\{\textbf{x},\textbf{y},\textbf{z}\}$ such that for each one of them, say $\textbf{u}$, this oriented edge lies between two successive levels of $\textbf{u}$. Thus there is a common vertex, say $\textbf{x}$, such that for the two oriented edges $e_1, e_2$ in $\mathcal{T}_n$, one lies between $a$-th level and $(a+1)$-th level of $\textbf{x}$, and the other lies between $a'$-th level and $(a'+1)$-th level of $\textbf{x}$. Since $[1/A,B]=-[A,B]$, we fix $\epsilon_{\textbf{x}}(e_1)=$ $\epsilon_{\textbf{x}}(e_2)=1$ without loss of generality. By symmetry, there are two cases as follows to be checked.
\begin{enumerate}
\item Suppose that $e_1=\overrightarrow{v_{\textbf{x},\textbf{y},\textbf{z}}^{a,b+1,c} v_{\textbf{x},\textbf{y},\textbf{z}}^{a+1,b,c}}$ and $e_2=\overrightarrow{v_{\textbf{x},\textbf{y},\textbf{z}}^{a',b'+1,c'} v_{\textbf{x},\textbf{y},\textbf{z}}^{a'+1,b',c'}}$ where the non-negative integers $a,b,c,a',b',c'$ satisfy $a+b+c=n-1$ and $a'+b'+c'=n-1$ as in Figure \ref{der1}(1)(3).
\begin{enumerate}
\item \label{case:drdr1}
If $b>b'$, we arrange $v_1,\cdots,v_n,u_1,\cdots,u_n$ so that $x_{n-1}\prec \cdots \prec x_1 \prec v_1\prec\cdots \prec v_n\prec y_{n-1}\prec\cdots \prec y_1 \prec z_{n-1}\prec\cdots \prec z_1\prec u_1\prec\cdots \prec u_n \prec x_{n-1}$ as in Figure \ref{Figure:swapother4} for using Proposition \ref{propxyz}. By Corollary \ref{corollary:PP}, the swapping bracket between two oriented edge ratios does not depend on the right side $n$-tuples $(v^n), (u^n)$ that we choose. By Proposition \ref{propxyz} with the condition (*) $l\geq l'$ there, we have
\begin{eqnarray*}
  &&\left[E\left(x^{a},y^{b},z^{c}|y_{b+1},x_{a+1}\right), E\left(x^{a'},y^{b'},z^{c'}|y_{b'+1},x_{a'+1}\right)\right]
  \\&=& \left[(-1)^{2c+b}\frac{\Delta\left((x^a,y^{b+1},z^c),(v^n)\right)}{\Delta\left((x^{a+1},y^{b},z^c),(v^n)\right)}, (-1)^{2c'+b'}\frac{\Delta\left((x^{a'},y^{b'+1},z^{c'}),(u^n)\right)}{\Delta\left((x^{a'+1},y^{b'},z^{c'}),(u^n)\right)}\right]
\\&=& \left[\frac{\Delta\left((x^a,y^{b+1},z^c),(v^n)\right)}{\Delta\left((x^{a+1},y^{b},z^c),(v^n)\right)}, \frac{\Delta\left((x^{a'},y^{b'+1},z^{c'}),(u^n)\right)}{\Delta\left((x^{a'+1},y^{b'},z^{c'}),(u^n)\right)}\right]
\\&=&\left[\Delta\left((x^a,y^{b+1},z^c),(v^n)\right), \Delta\left((x^{a'},y^{b'+1},z^{c'}),(u^n)\right)\right]
\\&-&\left[\Delta\left((x^a,y^{b+1},z^c),(v^n)\right), \Delta\left((x^{a'+1},y^{b'},z^{c'}),(u^n)\right)\right]
\\&-&\left[\Delta\left((x^{a+1},y^{b},z^c),(v^n)\right), \Delta\left((x^{a'},y^{b'+1},z^{c'}),(u^n)\right)\right]
\\&+& \left[\Delta\left((x^{a+1},y^{b},z^c),(v^n)\right), \Delta\left((x^{a'+1},y^{b'},z^{c'}),(u^n)\right)\right]
\\& =& \left(\frac{1}{2}\cdot\min\{a,a'\}- \frac{1}{2}\cdot\min\{b+1,b'+1\}-  \frac{1}{2}\cdot\min\{c,c'\})\right)
\\&-&\left(\frac{1}{2}\cdot\min\{a,a'+1\}- \frac{1}{2}\cdot\min\{b+1,b'\}-  \frac{1}{2}\cdot\min\{c,c'\})\right)
\\& -&\left(\frac{1}{2}\cdot\min\{a+1,a'\}- \frac{1}{2}\cdot\min\{b,b'+1\}-  \frac{1}{2}\cdot\min\{c,c'\})\right)
\\& +&\left(\frac{1}{2}\cdot\min\{a+1,a'+1\}- \frac{1}{2}\cdot\min\{b,b'\}-  \frac{1}{2}\cdot\min\{c,c'\})\right)
\\&=& \frac{1}{2} \cdot \left(\min\{a,a'\} - \min\{a,a'+1\} - \min\{a+1,a'\} + \min\{a+1,a'+1\}\right) 
\\&-& \frac{1}{2} \cdot \left(\min\{b,b'\} - \min\{b,b'+1\} - \min\{b+1,b'\} + \min\{b+1,b'+1\}\right)
\\&=& \frac{1}{2} \cdot \left(\min\{a,a'\} - \min\{a,a'+1\} - \min\{a+1,a'\} + \min\{a+1,a'+1\}\right)  
\\&=& \begin{cases}

   \frac{1}{2} &\mbox{if $a'=a$,}\\

   0 &\mbox{if $a'\neq a$.}
\end{cases}
\end{eqnarray*}
Note that in this case if $a'=a$, as in Figure \ref{der1}(1) $e_2$ is after $e_1$ for $\textbf{x}$.
\item If $b<b'$, we have $b'> b$. Using the computation in Case (\ref{case:drdr1}), we get
\begin{eqnarray*}
  &&\left[E\left(x^{a},y^{b},z^{c}|y_{b+1},x_{a+1}\right), E\left(x^{a'},y^{b'},z^{c'}|y_{b'+1},x_{a'+1}\right)\right] 
  \\&=& -\left[E\left(x^{a'},y^{b'},z^{c'}|y_{b'+1},x_{a'+1}\right), E\left(x^{a},y^{b},z^{c}|y_{b+1},x_{a+1}\right)\right]
  \\&=& \begin{cases}

   -\frac{1}{2} &\mbox{if $a'=a$,}\\

   0 &\mbox{if $a'\neq a$.}
\end{cases}
  \end{eqnarray*}
Note that in this case if $b<b'$ and $a'=a$, the oriented edge $e_2$ is before $e_1$ for $\textbf{x}$.

\item 
If $b=b'$ and $c< c'$, as in Figure \ref{der1}(3) $e_2$ is before $e_1$ for $\textbf{y}$ and $a\neq a'$. By Proposition \ref{propxyz} with the condition (*) $p\leq p'$ there, similarly we obtain 
\begin{eqnarray*}
&&\left[E\left(x^{a},y^{b},z^{c}|y_{b+1},x_{a+1}\right), E\left(x^{a'},y^{b'},z^{c'}|y_{b'+1},x_{a'+1}\right)\right]
\\&=& \frac{1}{2} \cdot \left(\min\{a,a'\} - \min\{a,a'+1\} - \min\{a+1,a'\} + \min\{a+1,a'+1\}\right) 
\\&-& \frac{1}{2} \cdot \left(\min\{b,b'\} - \min\{b,b'+1\} - \min\{b+1,b'\} + \min\{b+1,b'+1\}\right)
\\&=&-\frac{1}{2}.
\end{eqnarray*}
If $b=b'$ and $c> c'$, we have $e_2$ is after $e_1$ for $\textbf{y}$, $c'<c$ and $a\neq a'$. Using the above computation, we get 
\begin{eqnarray*}
&&\left[E\left(x^{a},y^{b},z^{c}|y_{b+1},x_{a+1}\right), E\left(x^{a'},y^{b'},z^{c'}|y_{b'+1},x_{a'+1}\right)\right]
\\&=&  -\left[E\left(x^{a'},y^{b'},z^{c'}|y_{b'+1},x_{a'+1}\right),E\left(x^{a},y^{b},z^{c}|y_{b+1},x_{a+1}\right)\right] =\frac{1}{2}.
\end{eqnarray*}
If $b=b'$ and $c=c'$, then $a=a'$. Thus in this case we have  
\[\left[E_{e_1}, E_{e_2}\right]=0.\] 
\end{enumerate}

\item Suppose that $e_1=\overrightarrow{v_{\textbf{x},\textbf{y},\textbf{z}}^{a,b,c+1} v_{\textbf{x},\textbf{y},\textbf{z}}^{a+1,b,c}}$ and $e_2=\overrightarrow{v_{\textbf{x},\textbf{y},\textbf{z}}^{a',b'+1,c'} v_{\textbf{x},\textbf{y},\textbf{z}}^{a'+1,b',c'}}$ where the non-negative integers $a,b,c,a',b',c'$ satisfy $a+b+c=n-1$ and $a'+b'+c'=n-1$ as in Figure \ref{der1}(2)(4).  
\begin{enumerate}
\item 
\label{case:drdr2}
If $b>b'$, by Proposition \ref{propxyz} with the condition (*) $l\geq l'$ there, we have
\begin{eqnarray*}
  &&\left[E\left(x^{a},y^{b},z^{c}|z_{c+1},x_{a+1}\right), E\left(x^{a'},y^{b'},z^{c'}|y_{b'+1},x_{a'+1}\right)\right]
\\&=& \left[(-1)^{b+c}\frac{\Delta\left((x^a,y^{b},z^{c+1}),(v^n)\right)}{\Delta\left((x^{a+1},y^{b},z^c),(v^n)\right)}, (-1)^{b'+2c'}\frac{\Delta\left((x^{a'},y^{b'+1},z^{c'}),(u^n)\right)}{\Delta\left((x^{a'+1},y^{b'},z^{c'}),(u^n)\right)}\right]
\\&=& \left[\frac{\Delta\left((x^a,y^{b},z^{c+1}),(v^n)\right)}{\Delta\left((x^{a+1},y^{b},z^c),(v^n)\right)}, \frac{\Delta\left((x^{a'},y^{b'+1},z^{c'}),(u^n)\right)}{\Delta\left((x^{a'+1},y^{b'},z^{c'}),(u^n)\right)}\right]
\\& =& \left(\frac{1}{2}\cdot\min\{a,a'\}- \frac{1}{2}\cdot\min\{b,b'+1\}-  \frac{1}{2}\cdot\min\{c+1,c'\})\right)
\\&-&\left(\frac{1}{2}\cdot\min\{a,a'+1\}- \frac{1}{2}\cdot\min\{b,b'\}-  \frac{1}{2}\cdot\min\{c+1,c'\})\right)
\\& -&\left(\frac{1}{2}\cdot\min\{a+1,a'\}- \frac{1}{2}\cdot\min\{b,b'+1\}-  \frac{1}{2}\cdot\min\{c,c'\})\right)
\\& +&\left(\frac{1}{2}\cdot\min\{a+1,a'+1\}- \frac{1}{2}\cdot\min\{b,b'\}-  \frac{1}{2}\cdot\min\{c,c'\})\right)
\\&=& \frac{1}{2} \cdot \left(\min\{a,a'\} - \min\{a,a'+1\} - \min\{a+1,a'\} + \min\{a+1,a'+1\}\right)  
\\&=& \begin{cases}

   \frac{1}{2} &\mbox{if $a'=a$,}\\

   0 &\mbox{if $a'\neq a$.}
\end{cases}
\end{eqnarray*}
Note that in this case if $a'=a$, then $e_2$ is after $e_1$ for $\textbf{x}$.

\item If $b\leq b'$, then $ b' \geq b$. Using the computation in Case (\ref{case:drdr2}), we get
\begin{eqnarray*}
  &&\left[E\left(x^{a},y^{b},z^{c}|z_{c+1},x_{a+1}\right), E\left(x^{a'},y^{b'},z^{c'}|y_{b'+1},x_{a'+1}\right)\right]
\\&=& - \left[E\left(x^{a'},y^{b'},z^{c'}|y_{b'+1},x_{a'+1}\right), E\left(x^{a},y^{b},z^{c}|z_{c+1},x_{a+1}\right)\right]
\\&=& -\frac{1}{2} \cdot \left(\left(\min\{a,a'\} - \min\{a,a'+1\} - \min\{a+1,a'\} + \min\{a+1,a'+1\}\right)  \right)
\\&=& \begin{cases}

   -\frac{1}{2} &\mbox{if $a'=a$,}\\

   0 &\mbox{if $a'\neq a$.}
\end{cases}
\end{eqnarray*}
Note that in this case if $a'=a$, then $e_2$ is before $e_1$ for $\textbf{x}$ as in Figure \ref{der1}(2).

\end{enumerate}
We conclude that for $\mathcal{N}=3$
\[\left[E_{e_1}, E_{e_2}\right]
=\begin{cases}\frac{1}{2} \cdot \epsilon_{\textbf{u}}(e_1) \cdot \epsilon_{\textbf{u}}(e_2) ,& e_2 \;\;is\;\;after\;\; e_1 \;\; for \;\; \textbf{u} \in \{\textbf{x},\textbf{y},\textbf{z}\}; \cr -\frac{1}{2}\cdot \epsilon_{\textbf{u}}(e_1) \cdot \epsilon_{\textbf{u}}(e_2) ,&e_2\;\;is\;\; before\;\; e_1 \;\; for \;\; \textbf{u}\in \{\textbf{x},\textbf{y},\textbf{z}\};\cr 0,&otherwise.\end{cases}
\]
\end{enumerate}

\begin{figure}
\includegraphics[scale=0.5]{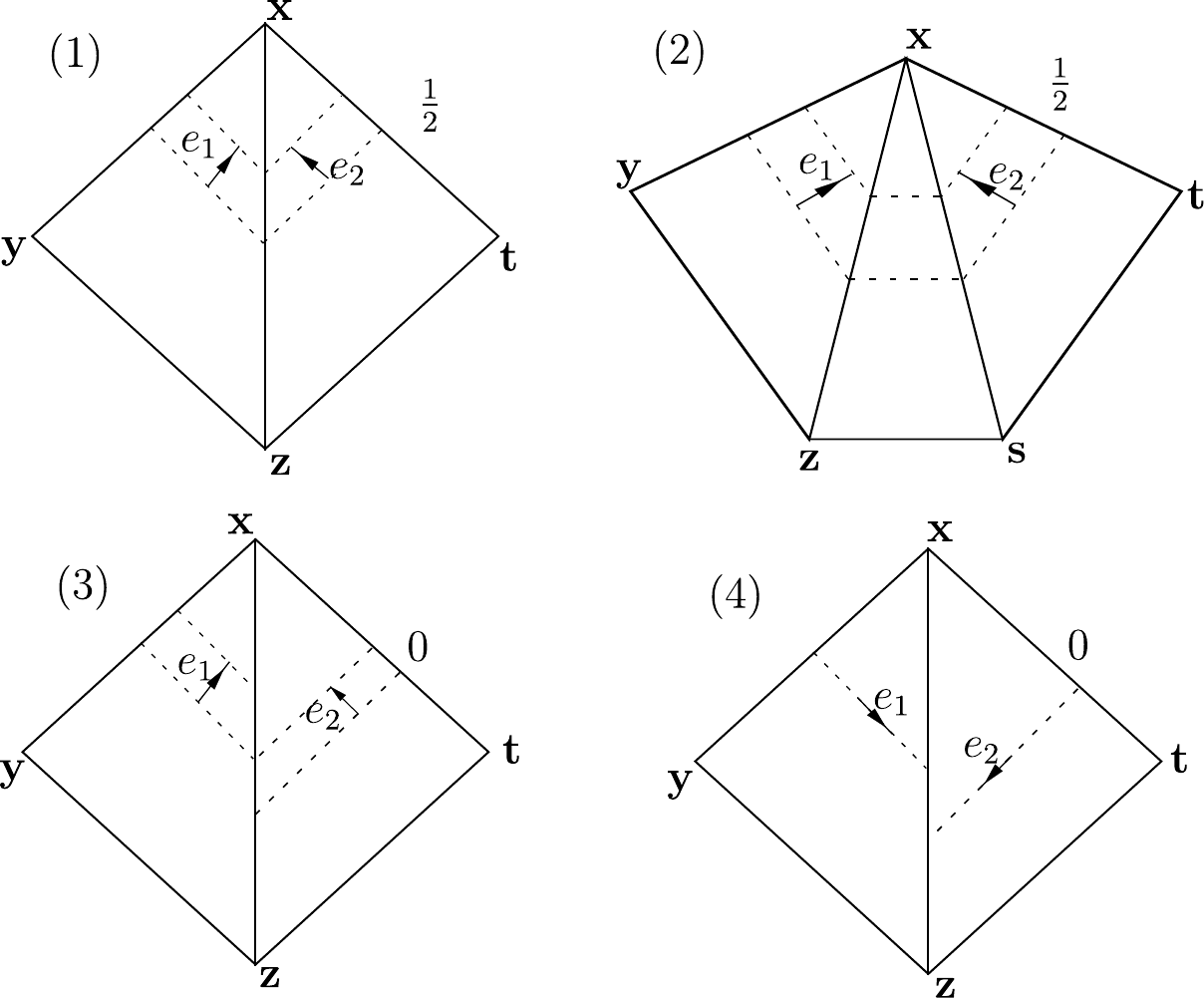}
\caption{Computation of $[E_{e_1},E_{e_2}]$ when the triangles for two oriented edge ratios have two common vertices or one common vertex.}
\label{der2}
\end{figure}

Suppose $\mathcal{N}=2$ as in Figure \ref{der2}(1)(3)(4). If $a\leq a'$ ($(a,b,c)$ and $(a',b',c')$ are defined similarly as the case $\mathcal{N}=3$ with respect to $(\textbf{x},\textbf{y},\textbf{z})$ and $(\textbf{x},\textbf{z},\textbf{t})$), we can combine the points $t_1,\cdots, t_{c'} \in \mathcal{P}$ with the points $x_1,\cdots, x_{a'}\in \mathcal{P}$ and apply Proposition \ref{propxyz} with the condition $l\geq l'=0$ there. If $a \geq a'$, we use $\left[E_{e_1}, E_{e_2}\right] = -\left[E_{e_2}, E_{e_1}\right]$ for arguing the same way as above. The case $\mathcal{N}=1$ follows in a similar way.
\end{proof}

\begin{proof}[\textbf{Proof of Theorem \ref{thmm}}]
To prove the theorem, we have to verify that
\[
\left[\theta_{\mathcal{T}_n}\left(X_A\right),\theta_{\mathcal{T}_n}\left(X_\tau\right)\right]= \frac{\theta_{\mathcal{T}_n} \left(\left\{X_A, X_\tau\right\}_{n}\right)}{\theta_{\mathcal{T}_n}\left(X_A\right)\cdot \theta_{\mathcal{T}_n}\left(X_\tau\right)} =  \varepsilon_{A \tau},
\]
for any $A, \tau \in \mathcal{I}_n' \cup \mathcal{J}_n$.

Each $V\in \mathcal{I}_n'$ is related to a graph with $4$-gon and an edge in the ideal triangulation $\mathcal{T}$, and each $V\in \mathcal{J}_n$ is related to a triangle in the ideal triangulation $\mathcal{T}$.
By symmetry, we have the following possible cases as shown in Figure \ref{Gall}:
\begin{figure}
\includegraphics[scale=0.5]{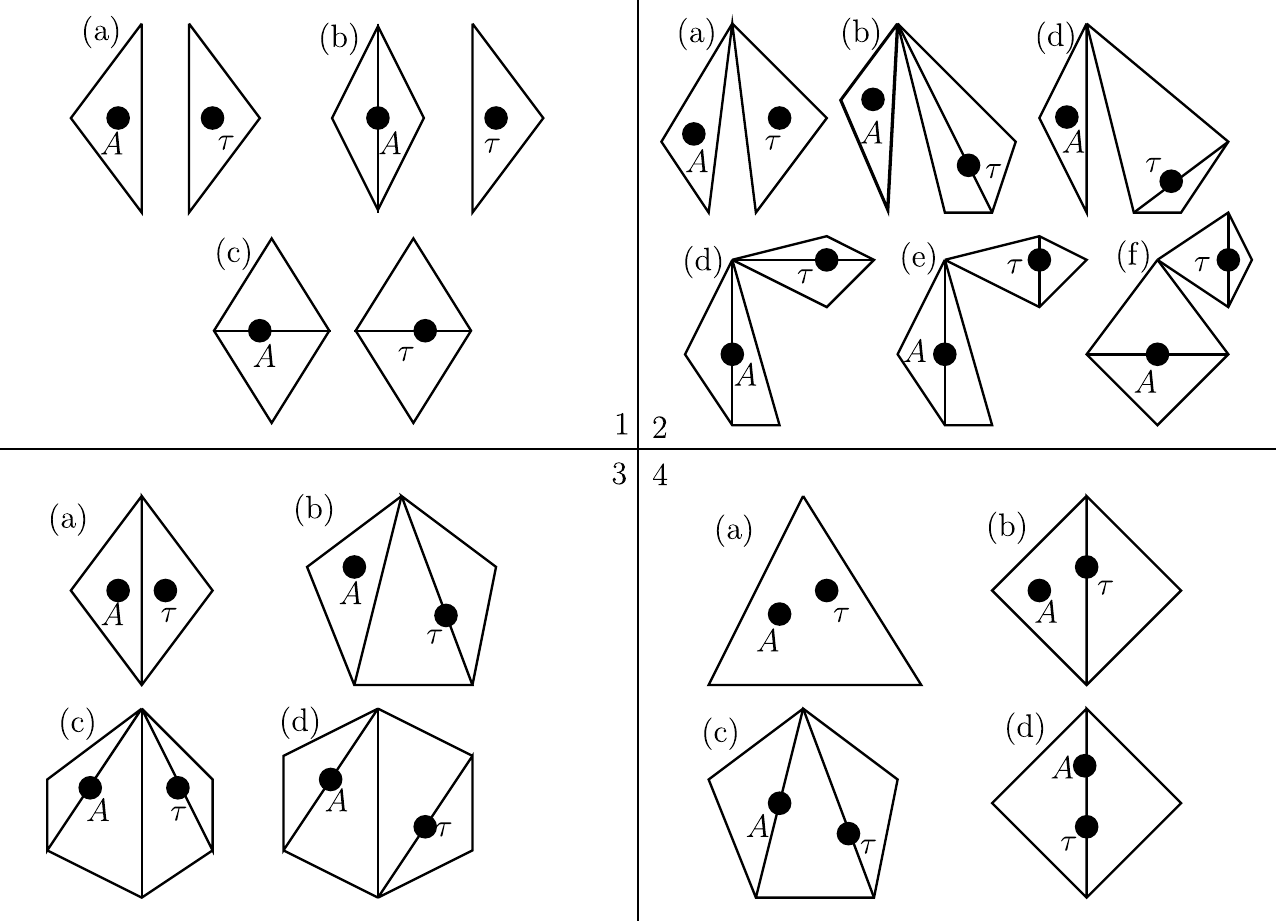}
\caption{All the cases up to symmetry}
\label{Gall}
\end{figure}
\begin{enumerate}
  \item Two graphs associated with $A$ and $\tau$ are separated by a line;
  \item Two graphs associated with $A$ and $\tau$ have one common point;
  \item Two graphs associated with $A$ and $\tau$ have two common points;
  \item other cases.
\end{enumerate}
Following Lemma \ref{lemct}, we write $\theta_{\mathcal{T}_n}\left(X_A\right)$ and $\theta_{\mathcal{T}_n}\left(X_\tau\right)$ as the products of directed edge ratios. Then we use Proposition \ref{propdr} to compute their swapping bracket case by case. We prove the case in Figure \ref{Gall} $4(a)$ explicitly and leave the details for the other cases to the reader.

Suppose $A \in \mathcal{J}_n$ ($\tau \in \mathcal{J}_n$ resp.) is specified by $(\textbf{x},\textbf{y},\textbf{z})$ and a triple of positive integers $(m,l,p)$ where $m+l+p=n$ ($(m',l',p')$ where $m'+l'+p'=n$ resp.). Let $e_1=\overrightarrow{v_{\textbf{x},\textbf{y},\textbf{z}}^{m+1,l,p-1} v_{\textbf{x},\textbf{y},\textbf{z}}^{m+1,l-1,p}}$, $e_2=\overrightarrow{v_{\textbf{x},\textbf{y},\textbf{z}}^{m-1,l+1,p} v_{\textbf{x},\textbf{y},\textbf{z}}^{m,l+1,p-1}}$, $e_3=\overrightarrow{v_{\textbf{x},\textbf{y},\textbf{z}}^{m,l-1,p+1} v_{\textbf{x},\textbf{y},\textbf{z}}^{m-1,l,p+1}}$, $e_1'=\overrightarrow{v_{\textbf{x},\textbf{y},\textbf{z}}^{m'+1,l',p'-1} v_{\textbf{x},\textbf{y},\textbf{z}}^{m'+1,l'-1,p'}}$, $e_2'=\overrightarrow{v_{\textbf{x},\textbf{y},\textbf{z}}^{m'-1,l'+1,p'} v_{\textbf{x},\textbf{y},\textbf{z}}^{m',l'+1,p'-1}}$, $e_3'=\overrightarrow{v_{\textbf{x},\textbf{y},\textbf{z}}^{m',l'-1,p'+1} v_{\textbf{x},\textbf{y},\textbf{z}}^{m'-1,l',p'+1}}$.
So
\[ \theta_{\mathcal{T}_n}\left(X_A\right) =  E_{e_1} \cdot
E_{e_2} \cdot
E_{e_3},\;\;\; \theta_{\mathcal{T}_n}\left(X_\tau\right) = E_{e_1'} \cdot
E_{e_2'} \cdot
E_{e_3'}.
\]
By the Leibniz's rule, we have
\begin{eqnarray*}
&&\left[\theta_{\mathcal{T}_n}\left(X_A\right) , \theta_{\mathcal{T}_n}\left(X_\tau\right)\right]
= \left[E_{e_1} E_{e_2} E_{e_3}, E_{e_1'} E_{e_2'} E_{e_3'}\right]
\\&=& \sum_{j=1}^3 \sum_{i=1}^3 \left[E_{e_i}, E_{e_j'}\right].
\end{eqnarray*}
Then we use Proposition \ref{propdr} for computing $\left[E_{e_i}, E_{e_j'}\right]$ for any $i,j=1,2,3$. Firstly, by observing Figure \ref{figure:strip}(1), we have 
\[\sum_{i=1}^3 \left[E_{e_i}, E_{e_j'}\right]= 0\]
if $e_j'$ lies outside the three strips formed by the dashed lines. Then, we compute case by case for $e_j'$ lying inside these stripes. We get the following property, denoted by $(**)$:

\emph{\[\sum_{i=1}^3 \left[E_{e_i}, E_{e_j'}\right]\neq 0
\] if and only if $e_j'$ lies strictly inside the hexagon formed by the vertices of $e_1,e_2,e_3$.
}
\begin{figure}
\includegraphics[scale=0.4]{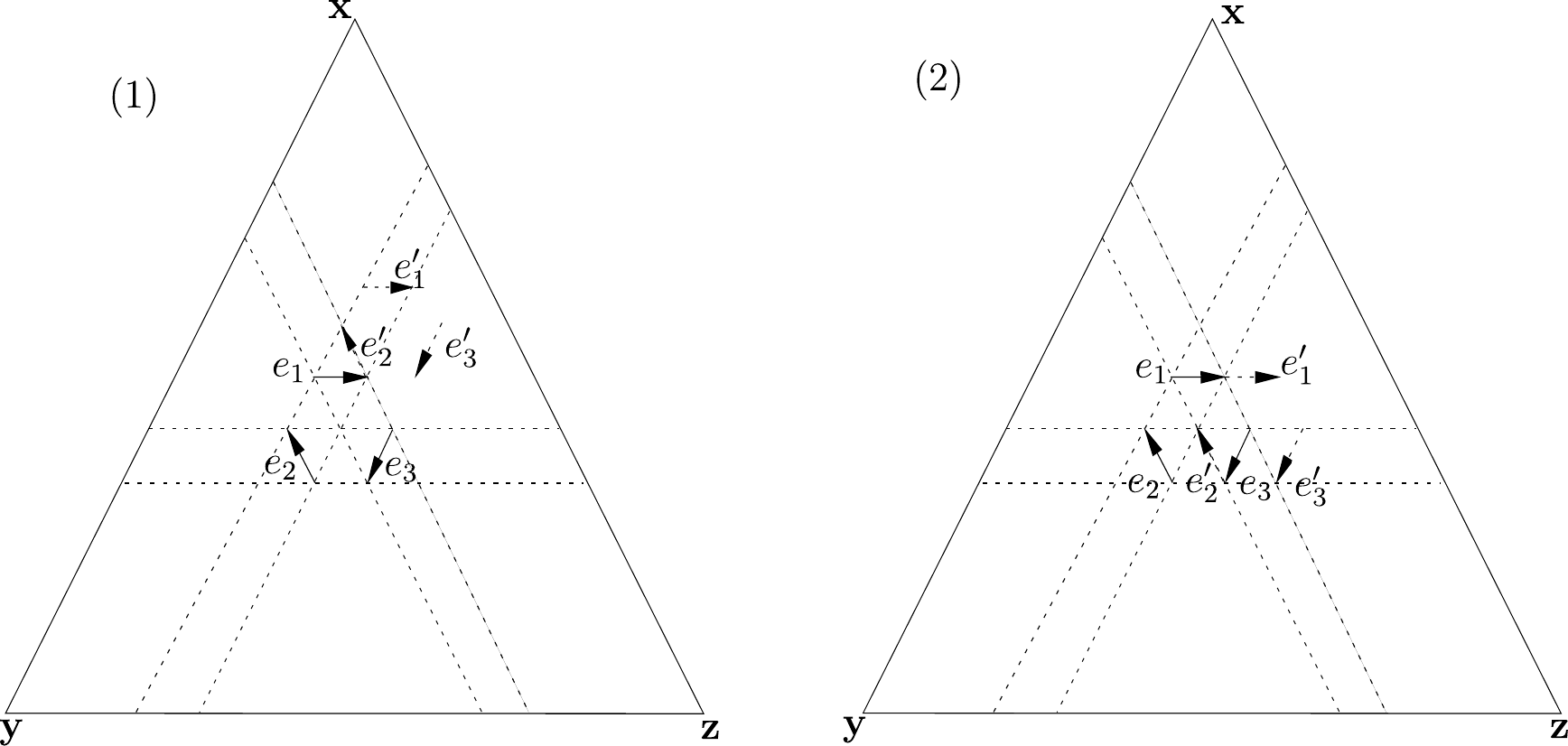}
\caption{Checking $\left[\theta_{\mathcal{T}_n}\left(X_A\right) , \theta_{\mathcal{T}_n}\left(X_\tau\right)\right]=\varepsilon_{A \tau}$.}
\label{figure:strip}
\end{figure}  
The property (**) provides us the only six possible edges for $e_j'$ that we have $\sum_{i=1}^3 \left[E_{e_i}, E_{e_j'}\right]\neq 0$.
Of course, when $(m,l,p) = (m',l',p')$, we have
\[\left[\theta_{\mathcal{T}_n}\left(X_A\right) , \theta_{\mathcal{T}_n}\left(X_\tau\right)\right] = 0 = \varepsilon_{A \tau}.\]
 When $(m,l,p) \neq (m',l',p')$, without loss of generality, suppose that $l>l'$.
\begin{enumerate}
  \item When $l-1>l'$, we have $l-1\geq l'+1$. As shown in Figure \ref{figure:strip}(1), none of $e_1',e_2',e_3'$ lie strictly inside the hexagon formed by the vertices of $e_1,e_2,e_3$. By property (**), we obtain
\[\left[\theta_{\mathcal{T}_n}\left(X_A\right) , \theta_{\mathcal{T}_n}\left(X_\tau\right)\right] = 0 = \varepsilon_{A \tau}.\]
  \item When $l-1=l'$:
  \begin{enumerate}
  \item When $m' \neq m$ and $m' \neq m+1$, none of $e_1',e_2',e_3'$ lie strictly inside the hexagon formed by the vertices of $e_1,e_2,e_3$. By property (**), we obtain
\[\left[\theta_{\mathcal{T}_n}\left(X_A\right) , \theta_{\mathcal{T}_n}\left(X_\tau\right)\right] = 0 = \varepsilon_{A \tau}.\]
  \item when $m' = m$, by Figure \ref{figure:strip}(2), we have
\begin{eqnarray*}
&&\left[\theta_{\mathcal{T}_n}\left(X_A\right) , \theta_{\mathcal{T}_n}\left(X_\tau\right)\right]
= \left[E_{e_1} E_{e_2} E_{e_3}, E_{e_1'} E_{e_2'} E_{e_3'}\right]
\\&=& \sum_{i=1}^3 \left[E_{e_i}, E_{e_1'}\right]+  \sum_{i=1}^3  \left[E_{e_i}, E_{e_2'}\right]+  \sum_{i=1}^3  \left[E_{e_i}, E_{e_3'}\right]
\\&=& 0 + \left(0+\frac{1}{2}+\frac{1}{2}\right)+0
\\&=& 1.
\end{eqnarray*}

\item Similarly, when $m' = m+1$, we have \[\left[\theta_{\mathcal{T}_n}\left(X_A\right) , \theta_{\mathcal{T}_n}\left(X_\tau\right)\right]=-1 .\]
  
  \end{enumerate}
In the cases of Figure \ref{Gall} 4(a), we conclude that
\[\left[\theta_{\mathcal{T}_n}\left(X_A\right) , \theta_{\mathcal{T}_n}\left(X_\tau\right)\right] = \varepsilon_{A \tau}.\]
     \end{enumerate}
\end{proof}

\section{Compatibility between any $\theta_{\mathcal{T}_n}$ and $\theta_{\mathcal{T}_n'}$}
For any two ideal triangulations $\mathcal{T}$ and $\mathcal{T}'$ of $D_k$, a finite sequence of flips can transform $\mathcal{T}$ to $\mathcal{T}'$. To prove that $\theta_{\mathcal{T}_n}$ and $\theta_{\mathcal{T}_n'}$ are compatible, it is enough to prove the compatibility when $\mathcal{T}'$ is obtained from $\mathcal{T}$ by one flip at the edge $e$. Let $\mathcal{T}_n$ and $\mathcal{T}_n'$ be the $n$-triangulations of $\mathcal{T}$ and $\mathcal{T}'$ respectively. As shown in \cite[Section 10.3]{FG06}, the flip $f_e$ is a composition of the mutations where each mutation corresponds to a Pl\"ucker relation for $\mathbb{K}^n$. Let us denote the transition map for the flip $f_e$ by $\mu_e^X$. It induces a rational map $\mu_e^{X*}$ from $\mathcal{FX}(\mathcal{T}_n')$ to $\mathcal{FX}(\mathcal{T}_n)$.
 
\begin{prop}
\label{proposition:T}
Let $\mathcal{T}$ and $\mathcal{T}'$ be two ideal triangulations of $D_k$ such that $\mathcal{T}'$ is obtained from $\mathcal{T}$ by a flip $f_e$ at the edge $e$. For any $Q \in \mathcal{FX}(\mathcal{T}_n')$, we have
\[\theta_{\mathcal{T}_n}\circ \mu_e^{X*}(Q) = \theta_{\mathcal{T}_n'}(Q).
\]
\end{prop}

\begin{proof}
Since the homomorphism $\theta_{\mathcal{T}_n}$ is induced from the homomorphism $\chi_n$, to prove the proposition, it is enough to prove the Pl\"ucker relation in $\mathcal{Z}_n(\mathcal{P})$.
\end{proof}

The proof of the following Pl\"ucker relation for the rank $n$ swapping algebra $\mathcal{Z}_n(\mathcal{P})$ also works for the Pl\"ucker relation for $\mathbb{K}^n$ even for the degenerate case. We are not aware of an appropriate reference for this fact, and provide a proof here.
\begin{lem}{\sc[Pl\"ucker relation]}
\label{lemma:Plucker}
For any $x^{n-2},a,b,c,d, u_1,\cdots,u_n \in \mathcal{P}$, we have the following equality in $\mathcal{Z}_n(\mathcal{P})$
\begin{equation}
\begin{aligned}
&\Delta\left(\left(x^{n-2},a,d\right),\left(u_1,\cdots,u_n\right)\right) \cdot \Delta\left(\left(x^{n-2},b,c\right),\left(u_1,\cdots,u_n\right)\right)
\\& + \Delta\left(\left(x^{n-2},a,b\right),\left(u_1,\cdots,u_n\right)\right)
 \cdot \Delta\left(\left(x^{n-2},c,d\right),\left(u_1,\cdots,u_n\right)\right)
\\& = \Delta\left(\left(x^{n-2},a,c\right),\left(u_1,\cdots,u_n\right)\right)
 \cdot \Delta\left(\left(x^{n-2},b,d\right),\left(u_1,\cdots,u_n\right)\right).
 \end{aligned}
\end{equation}
\end{lem}
\begin{proof}
Let $\widehat{u_j}:=u_1,\cdots,u_{j-1},u_{j+1},\cdots,u_n$. We have
\begin{equation}
\begin{aligned}
\label{equplu}
&\Delta\left(\left(x^{n-2},a,d\right),\left(u^n\right)\right) \cdot \Delta\left(\left(x^{n-2},b,c\right),\left(u^n\right)\right)
 \\&+ \Delta\left(\left(x^{n-2},a,b\right),\left(u^n\right)\right)
 \cdot \Delta\left(\left(x^{n-2},c,d\right),\left(u^n\right)\right)
\\& - \Delta\left(\left(x^{n-2},a,c\right),\left(u^n\right)\right)
 \cdot \Delta\left(\left(x^{n-2},b,d\right),\left(u^n\right)\right)
\\& = \sum_{j=1}^n   \Delta\left(\left(x^{n-2},a\right),\left(\widehat{u_j}\right)\right) \cdot 
\left\{ (-1)^{j+n} \cdot  d u_j  \cdot \Delta\left(\left(x^{n-2},b,c\right),\left(u^n\right)\right) 
\right.
\\& \left. +(-1)^{j+n} \cdot b u_j \cdot \Delta\left(\left(x^{n-2},c,d\right),\left(u^n\right)\right)  - (-1)^{j+n} \cdot  c u_j \cdot \Delta\left(\left(x^{n-2},b,d\right),\left(u^n\right)\right) \right\}
\end{aligned}
\end{equation}
For $i=1,\cdots, n-2$, we get
\[(-1)^{n+i}\cdot \Delta\left(\left(x^{n-2},a,x_i\right),\left(u^n\right)\right)= \sum_{j=1}^n (-1)^{i+j} \cdot   \Delta\left(\left(x^{n-2},a\right),\left(\widehat{u_j}\right)\right) \cdot x_i u_j =0.
\]
Using the above formula, we obtain that the right hand side of Equation \eqref{equplu} equals
\begin{equation*}
\begin{aligned}
&- \sum_{j=1}^n   \Delta\left(\left(x^{n-2},a\right),\left(\widehat{u_j}\right)\right) \cdot 
\\&
\left\{ \sum_{i=1}^{n-2} (-1)^{i+j} x_i u_j \cdot  \Delta\left(\left(x_1,\cdots,x_{i-1},x_{i+1} \cdots,x_{n-2},b,c,d\right),\left(u^n\right)\right) \right.
\\& \left.  
 +(-1)^{n-1+j} \cdot b u_j \cdot \Delta\left(\left(x^{n-2},c,d\right),\left(u^n\right)\right)  + (-1)^{n+j}\cdot c u_j \cdot \Delta\left(\left(x^{n-2},b,d\right),\left(u^n\right)\right) 
 \right. \\&
 \left. + (-1)^{n+1+j}\cdot d u_j  \cdot \Delta\left(\left(x^{n-2},b,c\right),\left(u^n\right)\right) \right\}
 \\&=  -\sum_{j=1}^n  \Delta\left(\left(x^{n-2},a\right),\left(\widehat{u_j}\right)\right) \cdot \Delta\left(\left(x^{n-2},b,c,d\right),\left(u^n, u_j\right)\right)
 \\&=0.
 \end{aligned}
\end{equation*}

\end{proof}

Since $\chi_n$ does not depend on the ideal triangulation, as a consequence, we again prove the following result.
\begin{cor}
\label{corollary:tri}
The rank $n$ Fock--Goncharov Poisson bracket $\{\cdot,\cdot\}_n$ does not depend on the ideal triangulation.
\end{cor}
\section{From cross fractions to $(n\times n)$-determinant ratios}
\label{section:opers}
We will relate a cross fraction to a product of two $(n\times n)$-determinant ratios through the following identification of points on circles.
\begin{defn}
For $n\geq 2$, let $\mathcal{P}=\{s\prec w \prec \cdots \prec t \prec s\}$ and 
\[\mathcal{P}_{n-1}= \{s_{n-1}\prec \cdots \prec s_{1} \prec w_{n-1} \prec \cdots \prec w_{1} \prec \cdots \prec t_{n-1} \prec  \cdots \prec t_{1} \prec s_{n-1}\},\]
where each $r\in \mathcal{P}$ corresponds to $n-1$ anticlockwise ordered points $r_{n-1},\dots, r_1$ nearby in $\mathcal{P}_{n-1}$ as in Figure \ref{figure:rtdr}(1)(2).
Let $\mathcal{RT}_n(\mathcal{P})$ be the sub fraction ring of $\mathcal{Q}_n(\mathcal{P})$ generated by all elements like $\frac{yx}{zx}$. Let $\mathcal{DR}_n(\mathcal{P}_{n-1})$ be the sub fraction ring of $\mathcal{Q}_n(\mathcal{P}_{n-1})$ generated by all elements like $E\left(x^{n-1}|y_1,z_1\right)$.  

The {\em homomorphism $\mu$} from $\mathcal{RT}_n(\mathcal{P})$ to $\mathcal{DR}_n(\mathcal{P}_{n-1})$ is defined by extending the following formula on the generators to $\mathcal{RT}_n(\mathcal{P})$ using Leibniz's rule
\[\mu\left(\frac{yx}{zx}\right) = E\left(x^{n-1}|y_1,z_1\right). \]
\end{defn}

\begin{prop}
The homomorphism $\mu$ is well-defined and injective.
\end{prop}
\begin{proof}
By Theorem \ref{thm:inv1} $B_{n\mathbb{K}}/S_{n\mathbb{K}}  \cong \mathcal{Z}_n(\mathcal{P})$ and $B_{n\mathbb{K}}'/S_{n\mathbb{K}}'  \cong \mathcal{Z}_n(\mathcal{P}_{n-1})$. We consider the geometric model instead. We embed $B_{n\mathbb{K}}/S_{n\mathbb{K}}$ into $B_{n\mathbb{K}}'/S_{n\mathbb{K}}'$ with respect to the homomorphism $\mu$ in the following way. We associate a vector for $B_{n\mathbb{K}}/S_{n\mathbb{K}}$ corresponding to $x$ on the left to a vector for $B_{n\mathbb{K}}'/S_{n\mathbb{K}}'$ corresponding to $x_1$ on the left,  a covector in $B_{n\mathbb{K}}/S_{n\mathbb{K}}$ corresponding to $x$ on the right to the wedge of $(n-1)$ vectors in $B_{n\mathbb{K}}'/S_{n\mathbb{K}}'$ corresponding to $x_1,\cdots,x_{n-1}$ on the left. Thus any relation in $\mathcal{RT}_n(\mathcal{P})$ generated by the $(n+1)\times (n+1)$ determinants in $R_n(\mathcal{P})$ is one-to-one correspondence with a relation in $\mu(\mathcal{RT}_n(\mathcal{P}))$ obtained from replacing each ordered pair by a $(n\times n)$-determinant with fixed right side $n$-tuple. We conclude that the homomorphism $\mu$ is well-defined and injective. 
\end{proof}

\begin{thm}
\label{theorem:cfd}
The injective homomorphism $\mu$ is Poisson with respect to the swapping bracket.
\end{thm}

\begin{figure}
\includegraphics[scale=0.5]{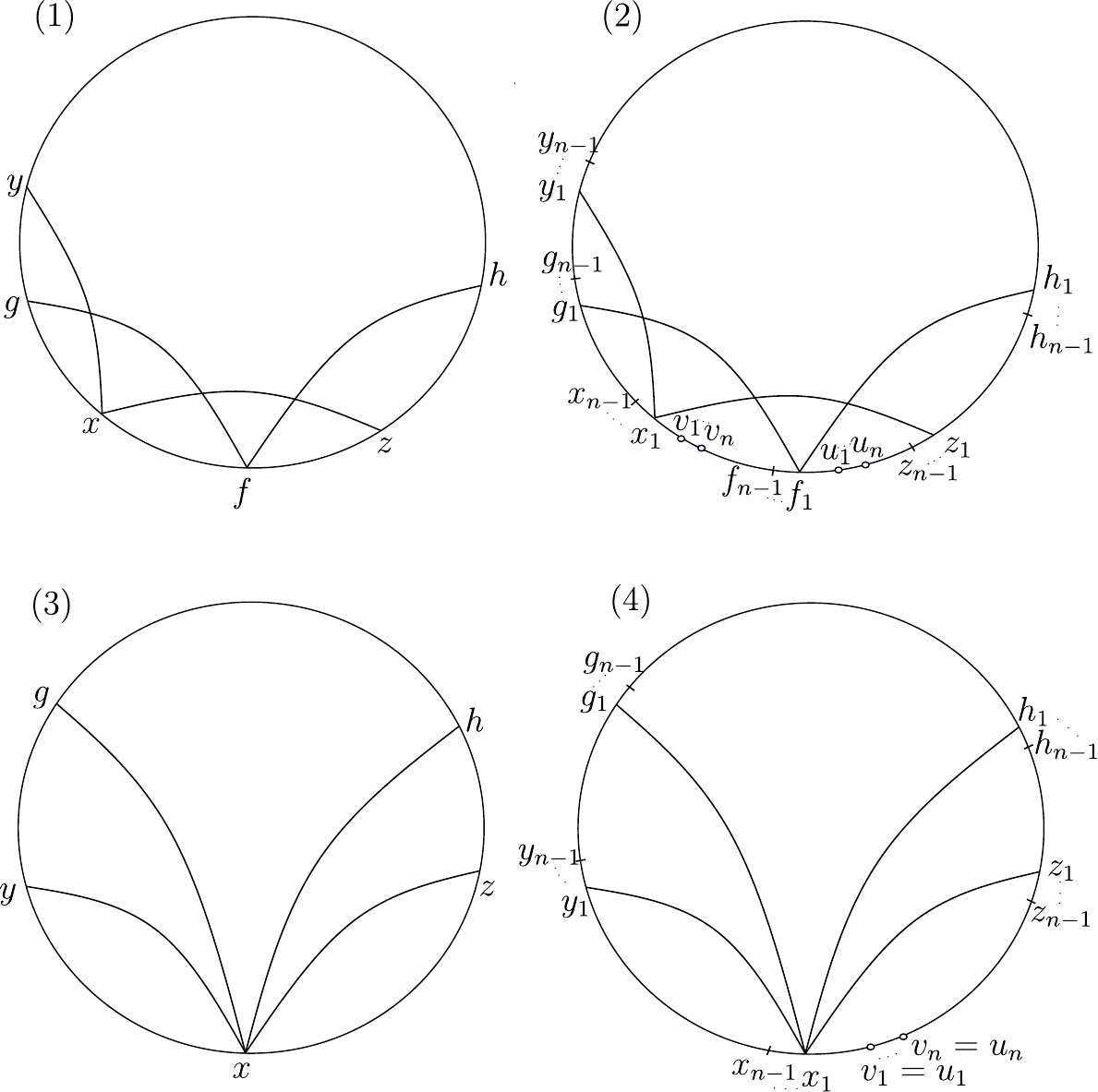}
\caption{}
\label{figure:rtdr}
\end{figure}

\begin{proof}
It is enough to prove the theorem for generators in $\mathcal{RT}_n(\mathcal{P})$. For two arbitrary generators $\frac{yx}{zx}, \frac{gf}{hf} \in \mathcal{RT}_n(\mathcal{P})$, we want to show that 
\[\mu\left(\left[\frac{yx}{zx}, \frac{gf}{hf}\right]\right)= \left[\mu\left(\frac{yx}{zx}\right), \mu\left(\frac{gf}{hf}\right)\right].\]
Firstly we have
\begin{equation}
\label{equation:mu}
\mu\left(\left[\frac{yx}{zx}, \frac{gf}{hf}\right]\right)=\mu([yx,gf])-\mu([yx,hf])-\mu([zx,gf])+\mu([zx,hf]).
\end{equation}
When $x\neq f$, we arrange $v^n$ ($u^n$ resp.) immediately after $x_1$ ($f_1$ resp.) with respect to the anticlockwise orientation as in Figure \ref{figure:rtdr}(2). Then we can use the fact that the linking number only depends on the corresponding position of the four points. This crucial arrangement allows us to get 
\begin{equation}
\label{equation:muid}
\begin{aligned}
&[\Delta((x^{n-1},y_1),(v^n)), \Delta((f^{n-1},g_1),(u^n))] 
\\&=\frac{\mathcal{J}(y_1 v_1, g_1 u_1)\cdot\Delta((x^{n-1},g_1),(v^n))\cdot \Delta((f^{n-1},y_1),(u^n))}{\Delta((x^{n-1},y_1),(v^n))\cdot \Delta((f^{n-1},g_1),(u^n))}
\\&= \mathcal{J}(y_1 x_1, g_1 f_1)\cdot E(x^{n-1}|g_1,y_1)\cdot E(f^{n-1}|y_1,g_1)
\\&= \mathcal{J}(y x, g f) \cdot \mu\left(\frac{gx}{yx}\right)\cdot \mu\left(\frac{yf}{gf}\right)
\\&= \mu\left([yx,gf]\right).
\end{aligned}
\end{equation}
We have similar formulas for the other three terms in the right hand side of Equation~\eqref{equation:mu}. Thus we obtain
\begin{equation}
\label{equation:mupoisson}
\begin{aligned}
&\mu\left(\left[\frac{yx}{zx}, \frac{gf}{hf}\right]\right)
\\&=[\Delta((x^{n-1},y_1),(v^n)), \Delta((f^{n-1},g_1),(u^n))]-[\Delta((x^{n-1},y_1),(v^n)), \Delta((f^{n-1},h_1),(u^n))]
\\&-[\Delta((x^{n-1},z_1),(v^n)), \Delta((f^{n-1},g_1),(u^n))]+[\Delta((x^{n-1},z_1),(v^n)), \Delta((f^{n-1},h_1),(u^n))]
\\&= \left[E(x^{n-1}|y_1,z_1), E(f^{n-1}|g_1,h_1)\right]= \left[\mu\left(\frac{yx}{zx}\right), \mu\left(\frac{gf}{hf}\right)\right].
\end{aligned}
\end{equation}
When $x=f$, we arrange the same way as above and $v_i=u_i$ for $i=1,\cdots,n$ as in Figure \ref{figure:rtdr}(4). By explicit computation, we get the same results as Equation \eqref{equation:muid} and Equation \eqref{equation:mupoisson} in this case. We conclude that the homomorphism $\mu$ is Poisson with respect to the swapping bracket.
\end{proof}

As a consequence, for $n=3$, our main theorem generalizes \cite[Chapter 3]{Su14} if we replace the elements like $E\left(x^{2}|y_1,z_1\right)$ by $\frac{yx}{zx}$. The advantage of the expression $\frac{yx}{zx}$ is that it allows us to get rid of $x_2 \in \mathcal{P}$, but it only works for $n\leq 3$. Using the Poisson homomorphism $\mu$, we have the following result.
\begin{cor}
The Poisson homomorphism in \cite[Theorem 10.7.2]{L18} is still Poisson after replacing each cross fraction by a product of two $(n\times n)$-determinant ratios.
\end{cor}

\section*{Acknowledgements}
This article generalizes the second part of my thesis under the direction of Fran\c cois Labourie at university of Paris-Sud. I very grateful to Fran\c cois Labourie for suggesting this subject and for his guidance.
I thank University of Paris-Sud, Max Planck Institute for Mathematics, Yau Mathematical Sciences Center at Tsinghua University and Luxembourg University.

%\bibliographystyle{amsalpha}
%\bibliography{Bibliography}

%\bibliographystyle{amsplain}

\end{document}